\documentclass{article}
\usepackage{amsmath,amssymb}
\usepackage{epsfig}
\usepackage{times}
\usepackage{graphicx}
\usepackage{authblk}

\usepackage{amsmath,amsthm}
\usepackage{amsfonts}
\usepackage{amssymb}
\usepackage{hyperref}
\usepackage{enumerate}
\usepackage{slashbox}
\usepackage{subfigure}
\usepackage{textcomp,enumerate}
\newtheorem{theorem}{Theorem}

\newtheorem{definition}[theorem]{Definition}
\newtheorem{example}[theorem]{Example}

\newtheorem{remark}[theorem]{Remark}
\usepackage{slashbox}
\usepackage{textcomp}
\usepackage{graphicx}

\providecommand{\keywords}[1]{\textbf{\textit{Index terms---}} #1}
\begin{document}



 \title{Blackwell-Nash Equilibrium for Discrete and Continuous Time Stochastic Games}

\author{Vikas Vikram Singh}
\affil{Laboratoire de Recherche en Informatique, Universit\'e Paris Sud, Orsay 91405, France\\
Email: vikas.singh@lri.fr, vikasstar@gmail.com}
\author{N. Hemachandra}
\affil{Industrial Engineering and Operations Research, Indian Institute of Technology Bombay,
Mumbai 400076, India\\
Email: nh@iitb.ac.in}

\maketitle
\begin{abstract}
We consider both discrete and continuous time finite state-action stochastic games. 
In discrete time stochastic games, it is known that a stationary Blackwell-Nash equilibrium (BNE)
exists for a single controller additive reward  (SC-AR) stochastic game which is a special case of a general  stochastic game. 
We show  that,  in general, the additive reward condition is needed for the existence of a
BNE. We give an example of a single controller stochastic game which does not satisfy additive reward condition.
We show that this example does not have a stationary BNE. 
For a general discrete time discounted stochastic game 
we give two different sets of conditions and show that 
a stationary Nash equilibrium that satisfies any set of conditions is 
 a BNE.   
One of these sets of conditions weakens a  set of conditions available in the literature.  
For  continuous time stochastic games, we give an example that does not have a stationary BNE.
In fact, this example is a single controller continuous time stochastic game. 
Then, we introduce a  continuous time  SC-AR stochastic game. We show that 
 there always exists a stationary deterministic BNE for  continuous time  SC-AR stochastic game. 
 For a general continuous time discounted stochastic game we give two different sets of conditions
 and show that  a Nash equilibrium that satisfies any set of conditions is a BNE.
\end{abstract}


\section{Introduction}
Blackwell optimality is a very desirable property of discrete time discounted Markov decision processes. It ensures the existence of
 an optimal policy  for every discount factor close
enough to one.  For finite state-action discrete time Markov decision process (DTMDP), Blackwell \cite{David} 
showed that a stationary deterministic Blackwell 
optimal policy always exists (see also \cite{Putterman}). 
It is natural to extend the concept of Blackwell optimality in discounted MDPs to discounted stochastic games.
Singh et al. \cite{Vikas2} call a strategy pair  Blackwell-Nash equilibrium (BNE)
if it is a Nash equilibrium for every discount factor close enough to one.
In discrete time discounted stochastic games there always exists a stationary Nash equilibrium 
for a fixed discount factor \cite{Fink}, \cite{Takahashi}.
However, the existence of a BNE  is not always guaranteed. This can be seen 
from ``Big Match'' stochastic game example \cite{Gillette}.
  There are some special classes of discrete time stochastic games admitting a stationary 
BNE.  Gimbert and  Zielonka
\cite{Gimbert} showed that a zero sum perfect information stochastic game
 with state dependent discount factors always possesses a stationary deterministic 
Blackwell optimal strategy pair.
Avrachenkov et al. \cite{Avrachenkov} proposed two algorithms to compute a Blackwell optimal strategy pair 
for two player zero sum
perfect information stochastic games. 
 Singh et al. \cite{Vikas2} proposed a single controller additive reward (SC-AR) stochastic game and 
 showed the existence of  a stationary deterministic BNE 
for a general sum SC-AR stochastic game.   
For general sum discounted stochastic games they proposed a set of three conditions which together are sufficient 
for any  of its stationary Nash equilibrium
to be a BNE.
There are not much results known for continuous time stochastic games.  
Recently, Neyman \cite{Neyman-CT} showed that there always exists 
a stationary Nash equilibrium for a finite state-action continuous time discounted stochastic game. 
 He showed the existence of a  stationary Blackwell optimal strategy pair for a two player continuous time zero sum perfect information 
stochastic game. 

In this paper, we consider both discrete and continuous time 2-player finite state-action stochastic games 
with discounted payoff criterion.
For discrete time stochastic game, we strengthen the BNE results given in \cite{Vikas2}.
We first show that the additive reward condition in  SC-AR stochastic game considered in \cite{Vikas2} is needed for the 
existence of BNE.
 We give an example of a 
single controller stochastic game which 
does not satisfy the 
additive reward condition.
We show that this example  does not admit a  stationary BNE. 
For general sum discounted stochastic games we weaken the conditions, given in \cite{Vikas2}, which together are sufficient
for any of its 
stationary Nash equilibrium to be a BNE. In particular, in \cite{Vikas2} the Markov chain induced by 
the Nash equilibrium  
is required to have only 
one absorbing state with all other states being transient; this seems to be a very strong condition. 
We now propose a weaker condition  where 
the Markov chain induced by the Nash equilibrium  satisfies state independent transition (SIT) property. 
 The condition on one period rewards
is suitably modified.
  We also propose another different set of conditions which together 
  are sufficient for a  stationary Nash equilibrium  to be 
a BNE. We now have two disjoint sets of conditions. Hence, it is clear that none of these sets of conditions 
are  necessary.    
Along similar lines of discrete time stochastic games we give the
BNE results for continuous time stochastic games. 
We first give an example that
shows that a stationary BNE need not always exist for a general continuous time 
stochastic game. In fact, the example belong to the class of single controller games.  
Hence, in general even a single controller continuous time stochastic game need not have a BNE.
Then, similar to discrete time stochastic game we introduce 
continuous time SC-AR stochastic game.
We show that there always exists a stationary deterministic 
BNE for continuous time SC-AR stochastic game.
For a general continuous time stochastic game we give two disjoint sets of  conditions such that a 
stationary Nash equilibrium satisfying any set of conditions is a BNE. 

We now describe the structure of the rest of our paper.  Section \ref{SG} contains the BNE results 
for 2-player discrete time stochastic games. Section \ref{CTSG} contains the  BNE results 
for 2-player continuous time stochastic games. We conclude the paper in Section \ref{Conclusions}.

\section{Discrete time stochastic games}\label{SG}
 We first consider discrete time stochastic games. 
We recall the details of the model like dynamics and notations from
\cite{Vikas2}.
 A 2-player stochastic game  is
described by the tuple $(S, A^1, A^2, r^1, r^2, p)$, where 
\begin{enumerate}
\item[(i)]  $S$ is a finite state space. Generic element of $S$ is denoted by $s$.
\item[(ii)] $A^{i}$ is a finite action set of player $i$, $i=1,2$, let $A^i(s)$ denotes the set of
actions available to player $i$ at state $s$, where
$A^i=\bigcup_{s\in S} A^i(s)$.
\item[(iii)]  For player $i$, $i=1,2$, $r^i:\mathcal{K} \rightarrow \mathbb{R}$ is immediate reward function, where 
$\mathcal{K}=~\big\{(s,a^1,a^2)| \ s\in S, \ a^1\in A^1(s), \ a^2\in A^2(s)\big\}$.
 
\item[(iv)]  For a given set $M$, let $\wp(M)$ be the set of all probability measures on $M$.
 The transition law of the game is the function $p:\mathcal{K}\rightarrow \wp(S)$.

\end{enumerate}
A stochastic game proceeds through stages $t=0,1,2,\cdots$. At  stage $t$ the game is in  state $s_t\in S$, 
player 1 chooses an action $a_t^1\in A^1(s_t)$ and 
player 2 chooses an action $a_t^2\in A^2(s_t)$, then player 1 (resp. player 2) receives 
immediate reward $r^1(s_t,a_t^1,a_t^2)$ \big(resp. $r^2(s_t,a_t^1,a_t^2)$\big).
At time $t+1$ game moves to state $s_{t+1}$ with probability $p(s_{t+1}|s_t,a_t^1,a_t^2)$. 
The same thing repeats at $s_{t+1}$  
and game continues for the infinite time horizon. Both the players are interested in maximizing their expected discounted 
reward collected during 
the play over infinite time horizon. 

Define a history at time $t$ as
$h_t=(s_0,a_0^1,a_0^2,s_1,a_1^1,a_1^2,\cdots,s_{t-1},a_{t-1}^1,a_{t-1}^2,s_t)$,
where $s_m\in S$ for $m=0,1,\cdots,t,$ and $a_m^i\in A^i(s_m)$ for $m=0,1,\cdots,t-1$, for all $i=1,2$.
Let $H_t$ denote the set of all possible histories of length $t$.
At time $t$ a decision rule $f_t$ of player 1 (resp., $g_t$ of player 2) assigns to each $h_t\in H_t$ with final state $s_t$ a 
probability measure $f_t(h_t)\in \wp(A^1(s_t))$ \big(resp., $g_t(h_t)\in \wp(A^2(s_t))$\big). 
A sequence of such decision rules is called  history dependent strategy of the game. 
A history dependent strategy is called  Markovian strategy if decision rule at time $t$ depends only on the state at time $t$.
A stationary strategy is a Markovian strategy which does not depend on the time, i.e.,
for a stationary strategy of player 1 (resp., player 2) there exists an $f$ (resp., $g$) such that
$f_t=f$ (resp., $g_t=g$) for all $t$. We denote, with
some abuse of notations, $f$ and
$g$ as stationary strategies of player 1 and player 2 respectively.
Let $F_S$ and $G_S$ denote the sets of all stationary strategies of player 1 and player 2 respectively.
A stationary strategy $f\in F_S$ is identified with $f=\left((f(1))^T,(f(2))^T,\cdots,(f\left(|S|\right))^T\right)^T$,
where for each $s\in S$, $f(s) \in \wp\big(A^1(s)\big)$;
$|M|$ denotes the cardinality of a given set $M$ and $T$ denotes the transposition.
Similarly, a stationary strategy $g\in G_S$ of player 2 is defined.
 It is well known that for a  discrete time stochastic game with discounted payoff criterion there always exists a stationary 
Nash equilibrium (see \cite{Fink}, \cite{Takahashi}). Therefore,  we restrict ourselves 
to stationary strategies.

For an initial state $s\in S$ and a strategy pair $(f,g)$ the expected discounted reward of player $i$, $i=1,2$, is defined as
\begin{equation}\label{disc_value}
 v_\beta^i(s,f,g)=\sum_{t=0}^\infty \beta^t [P^t(f,g)]_s r^i(f,g), 
\end{equation}
where $\beta\in [0,1)$ is a fixed discount factor, and  $P^0(f,g)$ is an identity matrix, 
and $P^t(f,g)$ is a $t$-step stochastic matrix induced by a
strategy pair $(f,g)$,  and $r^i(f,g)$ is a $|S|\times 1$ 
vector of the expected immediate rewards of player $i$ whose $s$th component is
$r^i(s,f,g)=~\sum_{a^1\in A^1(s)}\sum_{a^2\in A^2(s)}f(s,a^1)r^i(s,a^1,a^2)g(s,a^2)$. 
For a given matrix $B$, $[B]_k$ denotes its $k$th row. 
The expected discounted reward  defined by \eqref{disc_value} 
can be written as
$$ v_\beta^i(s,f,g)=[I-\beta P(f,g)]_s^{-1}r^i(f,g), \;\;\forall\; i=1,2,$$
where $I$ denotes an identity matrix.
A strategy pair $(f^*,g^*)\in F_S\times G_S$ is said to be a Nash equilibrium of a discounted stochastic game 
if for all $s\in S$  the following 
inequalities hold simultaneously:
\begin{gather*}
v_{\beta}^{1}(s, f^*,g^*)\geq v_{\beta}^{1}(s, f,g^*),\ \forall \ f\in F_S, \\
v_{\beta}^{2}(s, f^*,g^*)\geq v_{\beta}^{2}(s, f^*,g),\ \forall \ g\in G_S,.
\end{gather*}
It is possible to give the Nash equilibrium definition, in our setting, by 
 restricting to $f\in F_S$ and $g\in G_S$ because when one player's 
strategy is fixed to a stationary strategy, then 
 other player's problem is  a MDP where an optimal strategy exists in the space of stationary strategies.
We now introduce the notations which we use throughout this paper.
For $i=1,2$ and $s, s' \in S$,
\begin{itemize}
\item $R^{i}(s)=\left[r^{i}(s,a^{1},a^{2})\right]_{a^{1}=1,a^{2}=1}^{|A^1(s)|,|A^2(s)|}$, where $R^i(s)$ 
is the reward matrix of player $i$ at state $s$.

\item $v^{i}=\left(v^{i}(1),v^{i}(2),\cdots,v^{i}(|S|)\right)^{T}$.

\item $ P(s'|s)=\left[p(s'|s,a^{1},a^{2})\right]_{a^{1}=1,a^{2}=1}^{|A^1(s)|,|A^2(s)|}$.

\item $\textbf{1}_{n} = (1,1, \cdots, 1)^{T} \in \mathbb{R}^n$.
\end{itemize}

\subsection{Blackwell-Nash equilibrium in discrete time stochastic games} 
A strategy pair  is said to be a BNE if it is Nash equilibrium 
for all the discount factors close enough to one. 
We present some new results that strengthens the results given in \cite{Vikas2}. 
We  show that a stationary BNE
may not always exist if we relax additive reward assumption in SC-AR stochastic
games considered in \cite{Vikas2}.
We give an example of single controller stochastic game
that fails to satisfy the additive reward assumption. We show that there does not exist
a stationary BNE in this game.
For general stochastic games, we propose two disjoint
sets of conditions that are sufficient 
for a Nash equilibrium to be a BNE. One set of conditions are more general than the set of conditions
given in \cite{Vikas2}.
We recall the definition of BNE as given in \cite{Vikas2}. 
\begin{definition}[\cite{Vikas2}]
A strategy pair $(f^*,g^*)$ is said to be a BNE of a discrete time stochastic game if there exists a $\beta_0\in[0,1)$ such 
that $(f^*,g^*)$ is a \mbox{$\beta$-discounted} Nash equilibrium for every $\beta\in[\beta_0,1)$.
\end{definition}
%
%
%

\subsubsection{Single controller stochastic games} \label{SGCT}
In these games the transition probabilities are controlled by only one player. We assume that player 2 controls the transition probabilities, i.e.,
$p(s'|s,a^1,a^2) = p(s'|s,a^2)$ for all $s\in S$, $a^1\in A^1(s)$, $a^2\in A^2(s)$. Singh et al. \cite{Vikas2}  
further assume that the immediate rewards of player 1 satisfy additive condition \eqref{add_reward}  given below.
\begin{equation}\label{add_reward}
 r^1(s,a^1,a^2)=r_1^1(s,a^1)+r^1_2(s,a^2),\ \forall \ s\in S, a^1\in A^1(s), a^2\in A^2(s).
\end{equation}
They call these games  single controller additive reward (SC-AR) stochastic games.
Singh et al. \cite{Vikas2}  showed that there always exists a stationary 
deterministic Blackwell-Nash equilibrium for a SC-AR stochastic game.   We give an example which is a single controller game 
but does not satisfy the additive reward assumption. We show that this game does not have any stationary  Blackwell-Nash 
equilibrium. From this example it is clear that the stationary Blackwell-Nash equilibrium may not always exist 
in single controller stochastic games.

\begin{example} \label{example1}
We consider a  2 states stochastic game  where  both the players have two actions at state 1 and only one action at state 2, i.e.,
$S=\{1,2\}$, $A^1(1)=A^2(1)=~\{1,2\}$, $A^1(2)=A^2(2)=\{1\}$. The immediate rewards of both the players and the transition 
\noindent  probabilities for different combinations 
of states and actions are  summarized in the  Table \ref{Tb1}.  
\begin{table}[ht]\label{Table1}
\centering
\caption{Immediate rewards and Transition Probabilities}
\label{Tb1}
\subfigure[$s=1$]{\resizebox{0.6\linewidth}{.9cm}
{
\begin{tabular}{|r|r|}
\hline
\backslashbox{(4, 9)}{(1, 0)} & \backslashbox{(6, 3)} {(0, 1)}\\
\hline
\backslashbox{(5, 4)}{(1, 0)} &\backslashbox{(4, 5)}{(0, 1)}\\
\hline
\end{tabular}
}}
\subfigure[$s=2$]{\resizebox{0.3\linewidth}{.5cm}
{
\begin{tabular}{|r|}
\hline
\backslashbox{(6, 7)}{(1, 0)} \\
\hline
\end{tabular}.
}}
\end{table}
\end{example}
\noindent 
The rows and columns of the tables represent actions of player 1 and player 2 respectively.   
The upper half of each box of these tables represents transition probabilities 
and lower half represents immediate rewards. For example, if at state 1 both players choose their first action,
player 1 gets 4 and player 2 gets 9, and with probability 1 game  remains in state 1. 
From the above tables it is clear that the game is controlled only by player 2. 
The additive reward condition \eqref{add_reward} for Example \ref{example1} can be written as, 
\begin{align}
 &r^1_1(1,1)+r_2^1(1,1)=4 \label{eq1}\\
& r^1_1(1,1)+r_2^1(1,2)=6\label{eq2}\\
&r^1_1(1,2)+r_2^1(1,1)=5\label{eq3}\\
&r^1_1(1,2)+r_2^1(1,2)=4.\label{eq4}
\end{align}
It follows from  the subtraction  of \eqref{eq1} with \eqref{eq2} and the subtraction of \eqref{eq3} with \eqref{eq4}
that the above system of equations are inconsistent. That is,
the immediate rewards of 
player 1 are not additive.
 
\begin{theorem}\label{Single_Blackwell}
 The discrete time single controller stochastic game given in Example \ref{example1} does not have a stationary Blackwell-Nash equilibrium. 
\end{theorem}

\begin{proof}
We represent any stationary strategy pair $(f,g)= ((p, 1-p),(q, 1-q))$ for some $0\leq p,q\leq 1$ 
because at state 2 both the players have only one action. 
For a fixed stationary strategy of one player, the best response strategy of other player can be obtained by solving a DTMDP.
It is well known that in DTMDPs there always 
exists a stationary deterministic optimal strategy.   
For a fixed stationary strategy $g=(q, 1-q)$ of player 2, $f^*$ is a best response  of player 1 if and only if for each $s\in S$ 
$$v_\beta^1(s,f^*,g)=\max_{f\in F_S} v_\beta^1(s,f,g).$$ 
As the game is controlled only by player 2, so $f^*$ will be best response of player 1 if  and only if for each $s\in S$
\begin{equation}\label{disc_sing_beq1}
 r^1(s,f^*,g)=\max_{f\in F_S} r^1(s,f,g)= \max_{a^1\in A^1(s)} [R^1(s)g(s)]_{a^1}.
\end{equation}
We need to determine $f^*$ only at state $s=1$ because at $s=2$ there is only one action. We have
$$R^1(1)g(1)=[6-2q,4+q]^T.$$
Let $f_1=(1,0)$ and $f_2=(0,1)$ be two stationary deterministic strategies of player 1.
From \eqref{disc_sing_beq1}, we have

\begin{align}\label{disc_sing_beq2}
f^*= 
\begin{cases} f_1 &  \text{if} \ q<\frac{2}{3} 
\\
f_2 &\text{if} \ q>\frac{2}{3}\\
\big\{(p,1-p):  0\leq p\leq 1\big\} & \text{if} \ q=\frac{2}{3}.
\end{cases}
\end{align}
Equation \eqref{disc_sing_beq2} gives the best response of player 1, when player 2 fixes his strategy as $g=(q,1-q)$, for all $\beta\in [0,1)$.
  
For a fixed stationary strategy $f=(p,1-p)$ of player 1, player 2 faces a DTMDP with immediate rewards 
$\tilde{r}(1,1)=r^2(1,f,1)=4+5p$, $\tilde{r}(1,2)=r^2(1,f,2)=5-2p$, 
$\tilde{r}(2,1)=r^2(2,f,1)=7$ and the same transition probabilities as given in Example \ref{example1}.  Let $g_1=(1,0)$ and $g_2=(0,1)$ be two stationary 
deterministic strategies of player 2. By using the data given in Example \ref{example1} we have

\begin{equation}\label{disc_sing_beq3}
 v_\beta^2(g_1)=[I-\beta P(g_1)]^{-1}\tilde{r}(g_1)= \left[\frac{4+5p}{1-\beta},\frac{(4+5p)\beta}{1-\beta}+7\right]^T.
\end{equation}
\begin{equation}\label{disc_sing_beq4}
 v_\beta^2(g_2)=[I-\beta P(g_2)]^{-1}\tilde{r}(g_2)=\left[\frac{5-2p+7\beta}{1-\beta^2},\frac{(5-2p)\beta+7}{1-\beta^2}\right]^T.
\end{equation}
By using \eqref{disc_sing_beq3} and \eqref{disc_sing_beq4} we have
\begin{equation}\label{disc_sing_beq5}
 v_\beta^2(g_1)-v_\beta^2(g_2)=\left[\frac{p(7+5\beta)-(3\beta+1)}{1-\beta^2}, \frac{\beta(p(7+5\beta)-(3\beta+1))}{1-\beta^2}\right]^T.
\end{equation}
From \eqref{disc_sing_beq5} the best response $g^*$ of player 2 against a fixed strategy $f=(p,1-p)$ of player 1 for a given discount factor $\beta$
is given by \eqref{disc_sing_beq6}
 
\begin{align}\label{disc_sing_beq6}
g^*= 
\begin{cases} g_1 &  \text{if} \ p>\frac{3\beta+1}{7+5\beta} 
\\
g_2 &\text{if} \ p<\frac{3\beta+1}{7+5\beta}\\
\big\{(q,1-q):  0\leq q\leq 1\big\} & \text{if} \ p=\frac{3\beta+1}{7+5\beta}.
\end{cases}
\end{align}
 From \eqref{disc_sing_beq2}
and \eqref{disc_sing_beq6} it is easy to see that for a discount factor 
$\beta$, a strategy pair 
$(f_\beta^*,g_\beta^*)=\left(\left(\frac{3\beta+1}{7+5\beta},\frac{6+2\beta}{7+5\beta}\right),\left(\frac{2}{3},\frac{1}{3}\right)\right)$ 
is such that $f_\beta^*$ and $g_\beta^*$ are 
best responses of each other, i.e., it is a Nash equilibrium.  Next, we show that $(f_\beta^*,g_\beta^*)$ is the unique Nash equilibrium.
Let $(\tilde{f},\tilde{g})=((\tilde{p}, 1-\tilde{p}),(\tilde{q}, 1-\tilde{q}))$ for some $0\leq \tilde{p},\tilde{q} \leq 1$ 
be another Nash equilibrium different from $(f_\beta^*,g_\beta^*)$.
Now, we consider two cases. 

\noindent \textbf{Case I:} Let $\tilde{p}\neq \frac{3\beta+1}{7+5\beta}$. Then we have  two sub cases.
If $\tilde{p}>\frac{3\beta+1}{7+5\beta}$ then from \eqref{disc_sing_beq6} $\tilde{q}=1$. But, from \eqref{disc_sing_beq2} 
the best response of player 1 corresponding to
$\tilde{q}=1$ is $f_2$. This gives the contradiction because $f_2$ does not correspond to  $\tilde{p}>\frac{3\beta+1}{7+5\beta}$. If 
$\tilde{p}<\frac{3\beta+1}{7+5\beta}$   then from \eqref{disc_sing_beq6} $\tilde{q}=0$. 
But, from \eqref{disc_sing_beq2} the best response of player 1 corresponding to
$\tilde{q}=0$ is $f_1$. This again gives the contradiction because $f_1$ does not correspond to $\tilde{p}<\frac{3\beta+1}{7+5\beta}$.
Hence  $\tilde{p}\neq \frac{3\beta+1}{7+5\beta}$ is not possible.

\noindent \textbf{Case II:} If $\tilde{q}\neq \frac{2}{3}$. Again if $\tilde{q}>\frac{2}{3}$ then from \eqref{disc_sing_beq2} $\tilde{p}=0$. 
But, from \eqref{disc_sing_beq6} the best response of player 2 corresponding to  $\tilde{p}=0$ is $g_2$.  
This gives the contradiction from the similar argument given in
Case I. If $\tilde{q}<\frac{2}{3}$ the from \eqref{disc_sing_beq2} $\tilde{p}=1$. But, from \eqref{disc_sing_beq6} 
the best response of player 2 corresponding 
to  $\tilde{p}=1$ is $g_1$ which again gives the contradiction. Hence $\tilde{q}\neq \frac{2}{3}$ is not possible.

From Case I and Case II it is clear that $(f_\beta^*,g_\beta^*)$ is an unique Nash equilibrium for each $\beta$.
As $f_\beta^*$ is an invertible function of $\beta$, then the Nash equilibrium $(f_\beta^*,g_\beta^*)$ 
varies with discount factor $\beta$. This implies that Example \ref{example1} will not have a stationary Blackwell-Nash equilibrium. 
\end{proof}


\subsubsection*{Average Nash equilibrium}
Here we show that 
$\displaystyle\lim_{\beta\uparrow 1} 
(f_\beta^*,g_\beta^*)$=$\left(\left(\frac{1}{3},\frac{2}{3}\right),\left(\frac{2}{3},\frac{1}{3}\right)\right)$ 
is a Nash equilibrium for an average stochastic game. 
For a fixed stationary strategy $g=(q, 1-q)$ of player 2, $f^*$ is a best response  of player 1 if and only if for each $s\in S$ 
$$v_{ea}^1(s,f^*,g)=\max_{f\in F_S} v_{ea}^1(s,f,g)=\max_{f\in F_S} [P^*(g)]_s r^1(f,g),$$ 
where $P^*(g)$ is a Cesaro limit matrix of $P(g)$. So, $f^*$ will be best response of player 1 if  and only if for each $s\in S$
\begin{equation}\label{sing_avg_beq1}
 r^1(s,f^*,g)=\max_{f\in F_S} r^1(s,f,g)= \max_{a^1\in A^1(s)} [R^1(s)g(s)]_{a^1}.
\end{equation}
As similar to Theorem \ref{Single_Blackwell} the best response $f^*$ is given by \eqref{sing_avg_beq2}
\begin{align}\label{sing_avg_beq2}
f^*= 
\begin{cases} f_1 &  \text{if} \ q<\frac{2}{3} 
\\
f_2 &\text{if} \ q>\frac{2}{3}\\
\big\{(p,1-p):  0\leq p\leq 1\big\} & \text{if} \ q=\frac{2}{3}.
\end{cases}
\end{align}
Equation \eqref{sing_avg_beq2} gives the best response of player 1, when player 2 fixes his strategy as $g=(q,1-q)$.
For a fixed stationary strategy $f=(p,1-p)$ of player 1, player 2 faces a MDP with immediate rewards 
$\tilde{r}(1,1)=r^2(1,f,1)=4+5p$, $\tilde{r}(1,2)=r^2(1,f,2)=5-2p$, 
$\tilde{r}(2,1)=r^2(2,f,1)=7$ and the same transition probabilities as given in Example \ref{example1}.  By using the data given in Example \ref{example1} we have
\begin{equation}\label{sing_avg_beq3}
 v_{ea}^2(g_1)= P^*(g_1)\tilde{r}(g_1)= \left[4+5p,4+5p\right]^T.
\end{equation}
\begin{equation}\label{sing_avg_beq4}
 v_\beta^2(g_2)=P^*(g_2)\tilde{r}(g_2)=\left[6-p,6-p\right]^T.
\end{equation}
By using \eqref{sing_avg_beq3} and \eqref{sing_avg_beq4} we have
\begin{equation}\label{sing_avg_beq5}
 v_{ea}^2(g_1)-v_{ea}^2(g_2)=\left[6p-2,6p-2\right]^T.
\end{equation}
From \eqref{sing_avg_beq5} the best response $g^*$ of player 2 against a fixed strategy $f=(p,1-p)$ of player 1 
is given by \eqref{sing_avg_beq6}
 
\begin{align}\label{sing_avg_beq6}
g^*= 
\begin{cases} g_1 &  \text{if} \ p>\frac{1}{3} 
\\
g_2 &\text{if} \ p<\frac{1}{3}\\
\big\{(q,1-q):  0\leq q\leq 1\big\} & \text{if} \ p=\frac{1}{3}.
\end{cases}
\end{align}  
As similar to Theorem \ref{Single_Blackwell} we can show that  
$(f^*_{avg},g^*_{avg})= \left(\left(\frac{1}{3},\frac{2}{3}\right),\left(\frac{2}{3},\frac{1}{3}\right)\right)$ is an unique
Nash equilibrium for an undiscounted game. \\

\subsubsection*{Alternative method} We can also show the same result in a different way.
 Let 
$(f^*_{avg},g^*_{avg})= \left(\left(\frac{1}{3},\frac{2}{3}\right),\left(\frac{2}{3},\frac{1}{3}\right)\right)$. 
Define Car\big($f^*_\beta(s)$\big) = $\{a^1|f_\beta^*(s,a^1)>0\}$ and Car\big($g^*_\beta(s)$\big)=$\{a^2|g_\beta^*(s,a^2)>0\}$. It is easy to see that 
the  Car($f^*_\beta(s)$) and Car($g^*_\beta(s)$) are constant for all $\beta$ and for each $s$.
As $g^*_\beta=g^*_{avg}$ for all $\beta$ and the game is controlled only by player 2, then the Markov chain structure induced by 
$P(f^*_{avg},g^*_{avg})$ will be  same as the one induced by $P(f^*_\beta,g^*_\beta)$. Hence, from Corollary 5.3.9 of \cite{Filar}.

\begin{equation}\label{avg_Nasheq1}
v_{ea}^i(f^*_{avg},g^*_{avg})= \lim_{\beta\uparrow 1} (1-\beta) v_\beta^i(f_\beta^*,g_\beta^*), \ \forall \ i =1,2.
\end{equation}
For all $f\in F_S$, we have 
\begin{align}\label{avg_Nasheq2}
 v_{ea}^1(f,g_{avg}^*)&=\lim_{\beta\uparrow 1}(1-\beta) v_\beta^1(f,g_\beta^*)\nonumber\\
&\leq \lim_{\beta\uparrow 1} (1-\beta) v_\beta^1(f_\beta^*,g_\beta^*)\nonumber\\
&=v_{ea}^1(f^*_{avg},g^*_{avg}).
\end{align}
First equality above comes from the fact that $g^*_\beta=g^*_{avg}$ for all $\beta$ and then
inequality  comes from the fact that $(f_\beta^*,g_\beta^*)$ is a Nash equilibrium for all $\beta$. The last equality is due to \eqref{avg_Nasheq1}. 
For all $g\in G_S$, we have

\begin{align}\label{avg_Nasheq3}
 v_{ea}^2(f_{avg}^*,g)&=\lim_{\beta\uparrow 1} (1-\beta) v_\beta^2(f_\beta^*,g)\nonumber\\
&\leq \lim_{\beta\uparrow 1} (1-\beta)v_\beta^2(f_\beta^*,g_\beta^*)\nonumber\\
&=v_{ea}^2(f^*_{avg},g^*_{avg}).
\end{align}
First equality above comes from Corollary 5.3.9 of \cite{Filar} and then inequality  comes from the fact that $(f_\beta^*,g_\beta^*)$ is 
a Nash equilibrium for all $\beta$. The last equality is due to \eqref{avg_Nasheq1}. 
From \eqref{avg_Nasheq2} and \eqref{avg_Nasheq3} $(f^*_{avg},g^*_{avg})$ is a Nash equilibrium of an undiscounted stochastic game. 

%


\subsubsection{Sufficient conditions for Blackwell-Nash equilibrium}\label{Improved_Blackwell-cond}
 We give two different sets of conditions, i.e., a stationary Nash equilibrium that satisfies one set of 
conditions does not satisfy other set of conditions.  
As similar in \cite{Vikas2} we show that a stationary Nash equilibrium of a discounted stochastic game satisfying 
these sets of conditions is a Blackwell-Nash equilibrium. The set of conditions given in \cite{Vikas2}  come under  
a special case of the first set of conditions 
given here. Thus, these results are more general than given in \cite{Vikas2}. Since, there are two different sets of 
conditions for the existence 
of Blackwell-Nash equilibrium, then it is clear that
both sets of conditions are only sufficient but not necessary. 
For a Nash equilibrium $(f^*,g^*)$, 
 we denote $a_s^1\in A^1(s)$ (resp., $a_s^2\in A^2(s)$) as action of player 1
(resp., player 2) such that $f^*(s,a_s^1)=1$ (resp.,  $g^*(s,a_s^2)=1$) for all $s\in S$.
We use these notations throughout the paper.\\

\subsubsection*{First set of sufficient conditions}    
The first set of conditions are as follows:

 \noindent \textbf{C1}. $(f^*,g^*)$ is a pure strategy Nash equilibrium of a discounted stochastic game.

\noindent \textbf{C2}.
$$
P(f^*,g^*)=
\begin{matrix}
\begin{pmatrix}
  p_1 & p_2 &  \cdots & p_{|S|} \\
  p_1 & p_2 &  \cdots & p_{|S|} \\
  \vdots&  \vdots&  &\vdots\\
  p_1 & p_2 &  \cdots & p_{|S|}
 \end{pmatrix}
\end{matrix}
$$
where $p_s\geq 0, \ \forall \ s\in S, \ \sum_{s\in S}p_s=1$. The Markov chain induced by $(f^*,g^*)$ 
satisfies the state independent transition (SIT) property. 
\medskip

\noindent\textbf{C3}. 

$\left\{ \begin{array}{l}\displaystyle
\sum_{s'\in S} p_{s'} \ r^1(s',a_{s'}^1,a_{s'}^2)\geq \sum_{s'\in S} p(s'|s,a^1,a_s^2) r^1(s',a_{s'}^1,a_{s'}^2),
\; \forall\; s\in S, a^1\in A^1(s),\\ \vspace{.2cm}\\ \displaystyle
\sum_{s'\in S} p_{s'} \ r^2(s',a_{s'}^1,a_{s'}^2)\geq \sum_{s'\in S} p(s'|s,a_s^1,a^2) r^2(s',a_{s'}^1,a_{s'}^2),
\; \forall\; s\in S, a^2\in A^2(s).
\end{array}\right.$


\begin{remark}
If $p_{\bar{s}}=1$ for some $\bar{s}\in S$ and $p_s=0,\ \forall \ s\in S$, $s\neq \bar{s}$ then \textbf{C2} and \textbf{C3} correspond to 
the conditions given in \cite{Vikas2}. 
\end{remark}

\begin{remark}
 A pure Nash equilibrium of a SIT stochastic game will always satisfy the conditions \textbf{C1} and \textbf{C2}. 
\end{remark}

It is known that there is a one to one correspondence between the stationary Nash equilibria of a discounted stochastic game 
and the global minimizers, with objective function value zero, of a non-convex constrained optimization problem [OP] given below   
(see \cite{Thuijsman}, \cite{Filar}). We denote the decision variables and the objective function of [OP] by
$x=\left((v^1)^T,(v^2)^T,f^T,g^T\right)^T$ and $\psi(x)$ respectively.
{\allowdisplaybreaks
\begin{align*}
&\textup{{[OP]}} \quad
 \min_{x}\sum_{k=1}^{2}\textbf{1}_{|S|}^T\big[v^{k}-r^{k}(f,g)-\beta P(f,g)v^{k}\big] \\
&\text{s.t.}\\
& (i)~ R^{1}(s)g(s)+\beta\sum_{s'\in S}P(s'|s)g(s)v^{1}(s')
  \leq v^{1}(s)\textbf{1}_{|A^1(s)|}, \;\; \forall\; s\in S \\
& (ii)~ (f(s))^TR^{2}(s)+\beta\sum_{s'\in S}(f(s))^T P(s'|s)v^{2}(s')
 \leq  v^{2}(s)\textbf{1}^{T}_{|A^2(s)|}, \;\; \forall\; s\in S\\
& (iii)~ \sum_{a^1\in A^1(s)}f(s,a^{1})=1, \;\; \forall\; s\in S\\
& (iv)~ \sum_{a^2\in A^2(s)}g(s,a^{2})=1, \;\; \forall\; s\in S\\
& (v)~ f(s,a^{1}) \geq  0,  \;\; \forall\; s\in S,\; a^1\in A^1(s)\\
& (vi)~ g(s,a^{2}) \geq  0, \;\; \forall\; s\in S,\; a^2\in A^2(s).
\end{align*}
}
 
\begin{theorem}\label{Blackwell_disc_thm1}
 If $(f^*,g^*)$ is a stationary Nash equilibrium of a discrete time 
 discounted stochastic game at some discount factor $\hat{\beta}\in [0,1)$
and satisfies the conditions \textbf{C1}, \textbf{C2} and \textbf{C3}, then it will be a Blackwell-Nash equilibrium. 
\end{theorem} 
\begin{proof} We prove this by using the similar argument given in \cite{Vikas2}.
 Let $(f^*,g^*)$ be  a stationary Nash equilibrium of a  discounted stochastic game at some discount factor $\hat{\beta}$ 
and satisfies the conditions \textbf{C1}, \textbf{C2} and \textbf{C3}.
Define, the vector $v_\beta^{i*}=v_\beta^i(f^*,g^*)=\left[I-\beta P(f^*,g^*)\right]^{-1}r^i(f^*,g^*)$, $i=1,2$, where
\[ \left[I-\beta P(f^*,g^*)\right]^{-1}
=\frac{1}{1-\beta}
\begin{matrix}
\begin{pmatrix}
1-\beta+\beta p_1 ~& \beta p_2 &\cdots &\beta p_{|S|}\\
\beta p_1 ~& 1-\beta + \beta p_2 &\cdots &\beta p_{|S|}\\
\vdots&\vdots & & \vdots \\
\beta p_1 ~& \beta p_2 &\cdots & 1-\beta+\beta p_{|S|}
\end{pmatrix}
\end{matrix}.
\]
The value vector $v_\beta^{i*}, i=1,2$, can be written as a function of $\beta$ as,
\[
v_\beta^{i*}=
\begin{pmatrix}
v_\beta^{i*}(1)\\
v_\beta^{i*}(2)\\
\vdots \\
v_\beta^{i*}(|S|)
\end{pmatrix}
=
\begin{pmatrix}
 r^i(1,a_1^1,a_1^2)+\frac{\beta}{1-\beta}\sum_{s\in S}  p_s r^i(s,a_s^1,a_s^2) \\
 r^i(2,a_2^1,a_2^2)+\frac{\beta}{1-\beta}\sum_{s\in S}  p_s r^i(s,a_s^1,a_s^2)\\
\vdots \\
r^i(|S|,a_{|S|}^1,a_{|S|}^2)+\frac{\beta}{1-\beta}\sum_{s\in S}  p_s r^i(s,a_s^1,a_s^2)
\end{pmatrix}.
\]
Let $x^*= \left((v_\beta^{1*})^T,(v_\beta^{2*})^T,f^{*T},g^{*T}\right)^T$, then
 $\psi(x^*)=0$ for all $\beta\in [0,1)$. 
To show that $(f^*,g^*)$ is a
Blackwell-Nash equilibrium, it is sufficient  to show that there exists a $\beta_0\in [0,1)$ such that  $x^*$ is a
feasible point of the optimization problem [OP] for all $\beta\in [\beta_0,1)$.
At $x^*$ the constraints $(i)$ and $(ii)$ of [OP] can be written as
\[
r^1(s,a^1,a_s^2)+\beta\sum_{s'\in S} p(s'|s,a^1,a_s^2)v_\beta^{1*}(s')\leq v_\beta^{1*}(s),\;\; \forall \; s\in S,\; a^1\in A^1(s).
\]
\[
 r^2(s,a_s^1,a^2)+\beta\sum_{s'\in S} p(s'|s,a_s^1,a^2)v_\beta^{2*}(s')\leq v_\beta^{2*}(s),\;\; \forall\; s\in S,\; a^2\in A^2(s).
\]
For all $s\in S$, $a^1\in A^1(s)$, define
\[
 \theta_{s,a^1}^1=r^1(s,a^1,a_s^2)+\beta\sum_{s'\in S} p(s'|s,a^1,a_s^2)v_\beta^{1*}(s')- v_\beta^{1*}(s).
\]
For all $s\in S$, $a^2\in A^2(s)$, define
\[
 \theta_{s,a^2}^2= r^2(s,a_s^1,a^2)+\beta\sum_{s'\in S} p(s'|s,a_s^1,a^2)v_\beta^{2*}(s')- v_\beta^{2*}(s).
\]
Now, we consider two cases

\noindent\textbf{Case I:} ~
For each $s\in S$ we have two sub cases as given below.
\vspace{.2cm} 

\noindent If $a^1\in A^1(s)$ is such that $a^1=a_s^1$, then
\begin{align*}
 \theta_{s,a^1_s}^1&=r^1(s,a_s^1,a_s^2)+\beta \sum_{s'\in S} p_{s'} \ v_\beta^{1*}(s')-v_\beta^{1*}(s)\\
&=r^1(s,a_s^1,a_s^2)+ \beta \sum_{s'\in S}p_{s'}\left(r^1(s',a_{s'}^1,a_{s'}^2)+\frac{\beta}{1-\beta}\sum_{\tilde{s}\in S}  p_{\tilde{s}} \ 
r^1(\tilde{s},a_{\tilde{s}}^1,a_{\tilde{s}}^2)\right)\\
&\hspace{3cm}-r^1(s,a_s^1,a_s^2)-\frac{\beta}{1-\beta}\sum_{\tilde{s}\in S}  p_{\tilde{s}} \ r^1(\tilde{s},a_{\tilde{s}}^1,a_{\tilde{s}}^2)\\
&=0, \;\; \forall\; \beta.
\end{align*}
If $a^1\in A^1(s)$ is such that $a^1\neq a_s^1$, then,
\begin{align*}
&\theta_{s,a^1}^1=r^1(s,a^1,a_s^2)+\beta\sum_{s'\in S} p(s'|s,a^1,a_s^2)\left(r^1(s',a_{s'}^1,a_{s'}^2)
+\frac{\beta}{1-\beta}\sum_{\tilde{s}\in S}  p_{\tilde{s}} \ r^1(\tilde{s},a_{\tilde{s}}^1,a_{\tilde{s}}^2) \right)\\
&\hspace{4cm}-r^1(s,a_{s}^1,a_{s}^2)-\frac{\beta}{1-\beta}\sum_{\tilde{s}\in S}  p_{\tilde{s}} \ r^1(\tilde{s},a_{\tilde{s}}^1,a_{\tilde{s}}^2).
\end{align*}
That is
\begin{align}\label{Blackwell_e2_1}
\theta_{s,a^1}^1&=\left(r^1(s,a^1,a_s^2)-r^1(s,a_s^1,a_s^2)\right)\nonumber\\
&\hspace{2.5cm}-\beta\left(\sum_{s'\in S}  p_{s'} \ r^1(s',a_{s'}^1,a_{s'}^2)-\sum_{s'\in S}p(s'|s,a^1,a_s^2)r^1(s',a_{s'}^1,a_{s'}^2)\right).
\end{align}
When $\sum_{s'\in S}  p_{s'} \ r^1(s',a_{s'}^1,a_{s'}^2)=\sum_{s'\in S}p(s'|s,a^1,a_s^2) r^1(s',a_{s'}^1,a_{s'}^2)$, then, 
\eqref{Blackwell_e2_1} is independent of $\beta$.  Hence,
 $\theta_{s,a^1}^1\leq 0, \; \forall \; \beta$ because it holds for $\hat{\beta}$.
In other cases from \textbf{C3} we have, $\theta^1_{s,a^1}\leq 0, \;\; \forall\; \beta\geq \beta^1_{s,a^1}$,
where
\begin{equation}\label{Bl_bound1_1}
\beta_{s,a^1}^1=\frac{\left[r^1(s,a^1,a_s^2)-r^1(s,a_s^1,a_s^2)\right]}
{\left[\sum_{s'\in S}  p_{s'} \ r^1(s',a_{s'}^1,a_{s'}^2)-\sum_{s'\in S}p(s'|s,a^1,a_s^2)r^1(s',a_{s'}^1,a_{s'}^2)\right]}.
\end{equation}
It is clear that $\beta^1_{s,a^1}\leq\hat{\beta}<1$ because $(f^*,g^*)$ is a Nash equilibrium at $\hat{\beta}\in [0,1)$ 
and hence each constraint of [OP] is satisfied by $(f^*,g^*)$ at $\hat{\beta}$.

\noindent\textbf{Case II:} ~ 
As similar to Case I, for each $s\in S$ we have two sub cases.
\vspace{.2cm}

\noindent If $a^2\in A^2(s)$ is such that $a^2=a_s^2$, then, 
$\theta_{s,a_s^2}^2 = 0\;\; \forall \;\beta.$
If $a^2\in A^2(s)$ is such that $a^2\neq a_s^2$, then,
\begin{align}\label{Blackwell_e4_1}
 \theta_{s,a^2}^2&=\left(r^2(s,a_s^1,a^2)-r^2(s,a_s^1,a_s^2)\right)\nonumber\\
&\hspace{2.5cm}-\beta\left(\sum_{s'\in S}  p_{s'} \ r^2(s',a_{s'}^1,a_{s'}^2)
-\sum_{s'\in S} p(s'|s,a_s^1,a^2)
r^2(s',a_{s'}^1,a_{s'}^2)\right).
\end{align}
When $\sum_{s'\in S}  p_{s'} \ r^2(s',a_{s'}^1,a_{s'}^2)=\sum_{s'\in S} p(s'|s,a_s^1,a^2)r^2(s',a_{s'}^1,a_{s'}^2)$, 
then, \eqref{Blackwell_e4_1} is independent of $\beta$ and hence 
$\theta_{s,a^2}^2\leq 0, \; \forall \; \beta$.
In other cases from \textbf{C3} we have, $\theta_{s,a^2}^2\leq 0,\;\; \forall \; \beta\geq \beta_{s,a^2}^2$,
where
\begin{equation} \label{Bl_bound2_1}
\beta_{s,a^2}^2=\frac{\left[r^2(s,a_s^1,a^2)-r^2(s,a_s^1,a_s^2)\right]}{\left[\sum_{s'\in S}  p_{s'} \ r^2(s',a_{s'}^1,a_{s'}^2)
-\sum_{s'\in S} p(s'|s,a_s^1,a^2)r^2(s',a_{s'}^1,a_{s'}^2)\right]}.
\end{equation}
From the same argument as used in Case I, we have $\beta^2_{s,a^2}<1$. 
Now, define
\begin{equation}\label{Blackwell_bound_1}
 \beta_0=\max_{s\in S}\max_{a^1\in A^1(s);a^1\neq a^1_s}\max_{a^2\in A^2(s);a^2\neq a^2_s}\{0,\beta_{s,a^1}^1,\beta_{s,a^2}^2\}
\end{equation}
whenever $\beta_{s,a^1}^1$ and $\beta_{s,a^2}^2$ are well defined.
We include ``0'' in \eqref{Blackwell_bound_1} because lower bounds $\beta_{s,a^1}^1$, $\beta_{s,a^2}^2$ 
defined in \eqref{Bl_bound1_1}, \eqref{Bl_bound2_1} respectively can be negative also. It is clear that $\beta_0\in [0,1)$.
It is easy to see that the constraints $(i)$ and $(ii)$ of the optimization problem [OP] are feasible at
$x^*=\left((v_\beta^{1*})^T,(v_\beta^{2*})^T,f^{*T},g^{*T}\right)^T$ for all $\beta\in [\beta_0,1)$. 
The other constraints $(iii)$-$(vi)$ of [OP] does not depend on $\beta$ and are feasible at $(f^*,g^*)$. 
At $x^*$ the objective function value of [OP] is zero and all the constraints are feasible for all $\beta\in [\beta_0,1)$ 
which means that $(f^*,g^*)$ is a Blackwell-Nash equilibrium.
\end{proof}

\subsubsection*{Second set of sufficient conditions}
The second set of conditions are as follows:

 \noindent\textbf{D1}. $(f^*,g^*)$ is a pure strategy Nash equilibrium of a discounted stochastic game.

\noindent\textbf{D2}. The Markov chain induced by $(f^*,g^*)$ reduces into $|S|$  ergodic classes where each 
class contains only one state, i.e., 
$$
P(f^*,g^*)=
\begin{matrix}
\begin{pmatrix}
  1 & 0 &  \cdots & 0 \\
  0 & 1 &  \cdots & 0 \\
  \vdots&  \vdots&  &\vdots\\
  0 & 0 &  \cdots & 1
 \end{pmatrix}.
\end{matrix}
$$
\medskip 
\noindent\textbf{D3}.

$\left\{ \indent \begin{array}{l}\displaystyle
 r^1(s,a^1,a_{s}^2)\geq \sum_{s'\in S} p(s'|s,a^1,a_s^2) r^1(s',a_{s'}^1,a_{s'}^2),
\;\; \forall\; s\in S, a^1\in A^1(s),\\ \vspace{.2cm}\\\displaystyle
 r^2(s,a_{s}^1,a^2)\geq \sum_{s'\in S} p(s'|s,a_s^1,a^2) r^2(s',a_{s'}^1,a_{s'}^2),
\;\; \forall\; s\in S, a^2\in A^2(s).
\end{array}\right.$


\begin{theorem}\label{Blackwell_disc_thm2}
 If $(f^*,g^*)$ is a stationary Nash equilibrium of a discounted stochastic game at some discount factor $\hat{\beta}\in [0,1)$
and satisfies the conditions \textbf{D1}, \textbf{D2} and \textbf{D3}, then it will be a Blackwell-Nash equilibrium. 
\end{theorem}
\begin{proof}
We prove this by using the similar argument as in the proof of Theorem \ref{Blackwell_disc_thm1}.
  Let $(f^*,g^*)$ be  a stationary Nash equilibrium of a  discounted stochastic game at some discount factor $\hat{\beta}$ 
and satisfies the conditions \textbf{D1}, \textbf{D2} and \textbf{D3}.
Let $x^*=\big((v_\beta^{1*})^T,(v_\beta^{2*})^T,f^{*T},g^{*T}\big)^T$ where 
\[
 v_\beta^{i*}= [I-\beta P(f^*,g^*)]^{-1} r^i(f^*,g^*)=
\left(
\frac{r^i(1,a_1^1,a_1^2)}{1-\beta},  \\
 \frac{r^i(2,a_2^1,a_2^2)}{1-\beta},\\
\cdots, \\
\frac{r^i\big(|S|,a_{|S|}^1,a_{|S|}^2\big)}{1-\beta} 
\right),
\]
for $i=1,2$.
From the construction of the objective function of [OP], $\psi(x^*)=0$ for all $\beta\in [0,1)$. 
To show that $(f^*,g^*)$ is a
Blackwell-Nash equilibrium, it is sufficient to show that there exists a $\beta_0\in [0,1)$ such that  $x^*$ is a
feasible point of the optimization problem [OP] for all $\beta\in [\beta_0,1)$.  We discuss two cases. 

\noindent\textbf{Case I:} ~
For each $s\in S$ we have two sub cases as given below.
\vspace{.2cm} 

\noindent If $a^1\in A^1(s)$ is such that $a^1=a_s^1$, then,
 $\theta_{s,a^1_s}^1=0, \; \forall\; \beta$.
If $a^1\in A^1(s)$ is such that $a^1\neq a_s^1$, then,
\begin{align*}
&\theta_{s,a^1}^1=r^1(s,a^1,a_s^2)+\beta\sum_{s'\in S} p(s'|s,a^1,a_s^2)\frac{r^1(s',a_{s'}^1,a_{s'}^2)}{1-\beta}
-\frac{r^1(s,a_{s}^1,a_{s}^2)}{1-\beta}.
\end{align*}
That is,
\begin{align}\label{Blackwell_e2_2}
\theta_{s,a^1}^1&=\frac{1}{1-\beta}\left(r^1(s,a^1,a_s^2)-r^1(s,a_s^1,a_s^2)\right)\nonumber\\
&\hspace{3cm}-\frac{\beta}{1-\beta}\left(r^1(s,a^1,a_{s}^2)-\sum_{s'\in S}p(s'|s,a^1,a_s^2)r^1(s',a_{s'}^1,a_{s'}^2)\right).
\end{align}
When $r^1(s,a^1,a_{s}^2)=\sum_{s'\in S}p(s'|s,a^1,a_s^2) r^1(s',a_{s'}^1,a_{s'}^2)$, then, 
\eqref{Blackwell_e2_2} is independent of $\beta$ and hence
 $\theta_{s,a^1}^1\leq 0, \; \forall \; \beta$.
In other cases from \textbf{D3} we have, $\theta^1_{s,a^1}\leq 0, \;\; \forall\; \beta\geq \beta^1_{s,a^1}$,
where
\begin{equation}\label{Bl_bound1_2}
\beta_{s,a^1}^1=\frac{\left[r^1(s,a^1,a_s^2)-r^1(s,a_s^1,a_s^2)\right]}
{\left[r^1(s,a^1,a_{s}^2)-\sum_{s'\in S}p(s'|s,a^1,a_s^2)r^1(s',a_{s'}^1,a_{s'}^2)\right]}.
\end{equation}
It is clear that $\beta^1_{s,a^1}\leq\hat{\beta}<1$ because $(f^*,g^*)$ is a Nash equilibrium at $\hat{\beta}\in [0,1)$ 
and hence each constraint of [OP] is satisfied by $(f^*,g^*)$ at $\hat{\beta}$.

\noindent\textbf{Case II:} ~ 
As similar to Case I, for each $s\in S$ we have two sub cases.
\vspace{.2cm}

\noindent If $a^2\in A^2(s)$ is such that $a^2=a_s^2$, then, 
$\theta_{s,a_s^2}^2 = 0\;\; \forall \;\beta.$
If $a^2\in A^2(s)$ is such that $a^2\neq a_s^2$, then,

\begin{align}\label{Blackwell_e4_2}
 \theta_{s,a^2}^2&=\frac{1}{1-\beta}\left(r^2(s,a_s^1,a^2)-r^2(s,a_s^1,a_s^2)\right)\nonumber\\
&\hspace{3cm}-\frac{\beta}{1-\beta}\left(r^2(s,a_{s}^1,a^2)
-\sum_{s'\in S} p(s'|s,a_s^1,a^2)
r^2(s',a_{s'}^1,a_{s'}^2)\right).
\end{align}
When $r^2(s,a_{s}^1,a^2)=\sum_{s'\in S} p(s'|s,a_s^1,a^2)r^2(s',a_{s'}^1,a_{s'}^2)$, 
then, \eqref{Blackwell_e4_2} is independent of $\beta$ and hence 
$\theta_{s,a^2}^2\leq 0, \; \forall \; \beta$.
In other cases from \textbf{D3} we have, $\theta_{s,a^2}^2\leq 0,\;\; \forall \; \beta\geq \beta_{s,a^2}^2$,
where
\begin{equation} \label{Bl_bound2_2}
\beta_{s,a^2}^2=\frac{\left[r^2(s,a_s^1,a^2)-r^2(s,a_s^1,a_s^2)\right]}{\left[r^2(s,a_{s}^1,a^2)
-\sum_{s'\in S} p(s'|s,a_s^1,a^2)r^2(s',a_{s'}^1,a_{s'}^2)\right]}.
\end{equation}
From the same argument as used in Case I, we have $\beta^2_{s,a^2}<1$. 
Now, define
\begin{equation}\label{Blackwell_bound_2}
 \beta_0=\max_{s\in S}\max_{a^1\in A^1(s);a^1\neq a^1_s}\max_{a^2\in A^2(s);a^2\neq a^2_s}\{0,\beta_{s,a^1}^1,\beta_{s,a^2}^2\}
\end{equation}
whenever $\beta_{s,a^1}^1$ and $\beta_{s,a^2}^2$ are well defined.
Now, at $x^*$ the objective function value of [OP] is zero and the constraints are feasible for all $\beta\in [\beta_0,1)$ 
which means that $(f^*,g^*)$ is a Blackwell-Nash equilibrium.
\end{proof}


\noindent Next, we give an example where a stationary Nash equilibrium of the discounted game at $\beta=0.6$
 satisfies the conditions \textbf{D1}, \textbf{D2} and \textbf{D3} 
and hence it is a Blackwell-Nash equilibrium from Theorem \ref{Blackwell_disc_thm2}.

\begin{example}\label{Ex-sec-set}
We consider a  stochastic game where there are 2 states and  both the players have two 
actions at state 1 and  only one action at state 2 , i.e.,
$S=\{1,2\}$, $A^1(1)= A^2(1)= \{1,2\}$, $A^1(2)=A^2(2)=\{1\}$. 
The rewards of both the players and the transition probabilities for different combinations 
of states and actions are summarized in the Table \ref{Tb2}. 

\begin{table}[ht]
\centering
\caption{Immediate rewards and Transition Probabilities}
\label{Tb2}
\subfigure[$s=1$]{\resizebox{0.6\linewidth}{.9cm}
{
\begin{tabular}{|r|r|}
\hline
\backslashbox{(4,4.4)}{(1,0)} & \backslashbox{(4,5)} {(0,1)}\\
\hline
\backslashbox{(5,6)}{(0,1)} &\backslashbox{(3,2)}{(1,0)}\\
\hline
\end{tabular}
}}
\subfigure[$s=2$]{\resizebox{0.3\linewidth}{.5cm}
{
\begin{tabular}{|r|}
\hline
\backslashbox{(3,4)}{(0, 1)} \\
\hline
\end{tabular}.
}}
\end{table} 
\end{example}
\noindent 
In the above game there are two stationary deterministic strategies for each player.
We denote the stationary deterministic strategies of player 1 by $f_1=(1,0)$ and $f_2=(0,1)$
and the stationary deterministic strategies of player 2 by $g_1=(1,0)$ and $g_2=(0,1)$.  
We prove that  $(f_1,g_1)=\left((1,0), (1,0)\right)$ is a Blackwell-Nash equilibrium of the game. We first show that $(f_1,g_1)$ is a 
Nash equilibrium at $\beta=0.6$. Using the data given in Table \ref{Tb2}, we have 
$v_{0.6}^1(f_1,g_1)=(10,7.5)$, $v_{0.6}^2(f_1,g_1)=(11,10)$, $v_{0.6}^1(f_2,g_1)=(9.5,7.5)$, $v_{0.6}^2(f_1,g_2)=(11,10)$.
That is 
\begin{equation}\label{Ex_e5}
 v_{0.6}^1(f_1,g_1)\geq v_{0.6}^1(f_2,g_1).
\end{equation}
\begin{equation}\label{Ex_e6}
 v_{0.6}^2(f_1,g_1)\geq v_{0.6}^2(f_1,g_2).
\end{equation}
From \eqref{Ex_e5} and \eqref{Ex_e6} $(f_1,g_1)$ is a Nash equilibrium 
because for a fixed stationary strategy of one player, other player faces a 
MDP where optimal strategy exists in the space of stationary 
deterministic strategies (see \cite{Putterman}). 
 It is easy to verify that \textbf{D1}, \textbf{D2} and \textbf{D3}
hold at $(f_1,g_1)$, i.e., it is a Blackwell-Nash equilibrium. From  \eqref{Blackwell_bound_2},
$\beta_0=0.6$, so $(f_1,g_1)$ is a Nash equilibrium 
for all $\beta\in [0.6,1).$

\section{Continuous time stochastic game} \label{CTSG}
 We recall the definition of a continuous time stochastic game from \cite{Neyman-CT}.
  Similar to a discrete time stochastic game,  $S$ denote a finite set of states, 
and  $A^i(s)$ denote  a finite set of actions of player $i$ 
   available at state $s\in S$, and $r^i$ is an immediate payoff function of player $i$. 
  For all $s, s'\in S$ such that $s'\ne s$, and $a^1\in A^1(s)$, $a^2\in A^2(s)$,
 let $\mu(s',s,a^1,a^2)\ge 0$ be a rate of transition
  from state $s$ to state $s'$,  when player 1 and player 2 choose actions $a^1$ and $a^2$ respectively.
Denote $\mu(s,s,a^1,a^2)=-\sum_{s'\neq s} \mu(s',s,a^1,a^2)$.
 At time $t\in [0,\infty)$, if state is $s$, and player 1 plays an action $a^1$, and player 2 plays an action 
 $a^2$ during  infinitesimal time $dt$,
the payoff of player~1 is $r^1(s,a^1,a^2)dt$, and the payoff of player 2 is $r^2(s,a^1,a^2)dt$. 
  A transition  from $s$ to $s'$
occurs with
probability $\mu(s',s,a^1,a^2) dt $. It stays in state $s$ with probability $1+\mu(s,s,a^1,a^2)dt$. 
In the former case, the sojourn time at state $s$ follows an exponential distribution with parameter $-\mu(s,s,a^1,a^2)\ge 0$. 
  
A play of a continuous time stochastic game is a measurable function $h:[0,\infty)\rightarrow S\times \wp(A^1)\times \wp(A^2)$,
$t\mapsto h(t)=(s_t,x^1_t,x^2_t)$, with $x_t^i\in \wp(A^i(s_t))$, $i=1,2$.  
Given a play $h$, we define $h_t$ as history up to time $t$ as the restriction of the first coordinate of 
$h$  to the time interval $[0,t]$ and the restriction of the second and third coordinate to $[0,t)$.  The above 
definitions of play and  history 
in continuous time stochastic game are due to Neyman \cite{Neyman-CT}
where players observe their past mixed actions unlike the pure actions in discrete games. The decision of choosing action 
at any time $t$ might depend on various factors and it leads to different class of strategies. 
The case where decision of choosing  an action at any time $t$ depends on the entire history up to time $t$  
defines the history dependent strategies while for Markov strategies decision making  
 depend only on time $t$ and the state at time $t$. 
 The  stationary strategies are defined by the decision making rules 
that depend only on the states. The definition of stationary strategy $f$ (resp., $g$) 
of player 1 (resp., player 2) is same as in discrete time stochastic game.
Unlike in discrete time stochastic games, a strategy pair $(\pi_1,\pi_2)$ and an initial state $s_0$ 
need not define unambiguously a probability 
distribution $\mathbb{P}_{\pi_1,\pi_2}^{s_0}$
over plays of continuous time stochastic game.  A strategy profile $(\pi_1,\pi_2)$ is an admissible strategy profile, if for a given 
initial state $s_0$,  probability 
distributions $\mathbb{P}_{\pi_1',\pi_2}^{s_0}$ and $\mathbb{P}_{\pi_1,\pi_2'}^{s_0}$ 
over plays of continuous time stochastic game are unambiguously defined for all $\pi_1'$ and $\pi_2'$.
The class of Markov strategies and stationary strategies are contained in the class of
admissible strategies.
For a stationary strategy pair $(f,g)\in F_S\times G_S$ and initial state $s_0$, a unique probability 
distribution $P_{f,g}^{s_0}$ satisfies the equality,
 $$P^{s_0}_{f,g}(s_{t+\delta}=s|h_t)=\delta \mu(s_0,s_t,f(s_t),g(s_t))+o(\delta),$$
where, $\mu(s_0,s_t,f(s_t),g(s_t))=\sum_{a^1\in A^1(s_t)}\sum_{a^2\in A^2(s_t)}\mu(s_0,s_t,a^1,a^2)f(s_t,a^1)g(s_t,a^2)$.
For the details about all the definitions given above see \cite{Neyman-CT}. 
The existence of a stationary Nash equilibrium in the discounted continuous time stochastic 
game restricted to Markov strategies appears in \cite{Guo2}.
Neyman \cite{Neyman-CT} showed the existence of stationary Nash equilibrium 
by allowing history dependent strategies.
 Therefore, from now onwards we restrict ourselves to
the class of stationary strategies. 
 
 For a given strategy pair $(f,g)$ and an initial state $s$, the expected discounted reward of player $i$, $i=1,2$, is given by 
 \begin{equation}
  v_\alpha^i(s,f,g)=E_{f,g}^s\int_0^\infty e^{-\alpha t}r^i(s_t,x_t^1,x_t^2) dt,
 \end{equation}
where $\alpha>0$ is a discount rate. 
 A strategy pair $(f^*,g^*)$ is said to be an $\alpha$-discounted Nash equilibrium if for all $s\in S$  
  the following inequalities 
 hold,
 \begin{gather*}
  v_\alpha^1(s,f^*,g^*)\geq v_\alpha^1(s,f,g^*),\ \forall \ f\in F_S,\\
  v_\alpha^2(s,f^*,g^*)\geq v_\alpha^2(s,f^*,g),\ \forall \ g\in G_S.
 \end{gather*}
We call strategy pair $(f^*,g^*)$   Nash equilibrium  despite restricting $f$ and $g$ as stationary strategies. 
 This is possible because for a fixed stationary strategy of 
 one player, other player's problem is a continuous time Markov decision process (CTMDP)
 where an optimal strategy exists in the space of stationary  strategies \cite{Putterman}, \cite{Guo}. 
 \subsection{Some preliminary results and notations}
  We give some preliminary results which are useful in the subsequent analysis. 
  Define, 
   \[
||\mu||=\sup_{s\in S, a^1\in A^1(s),a^2\in A^2(s)}\left(\sum_{s'\in S;s'\neq s}\mu(s',s,a^1,a^2)\right).  
 \]
  For a fixed stationary strategy $g$ of player 2, player 1 faces 
 a CTMDP($g$). The immediate rewards and transition rates of CTMDP($g$) are respectively given by 
 $r^1(s,a^1,g)=\sum_{a^2\in A^2(s)} r^1(s,a^1,a^2)g(s,a^2)$
 and $\mu(s',s,a^1,g)=\sum_{a^2\in A^2(s)} \mu(s',s,a^1,a^2) g(s,a^2)$
 for all $s,s'\in S$, $a^1\in A^1(s)$. 
 It is well known that using uniformization technique  a CTMDP can be solved by an equivalent DTMDP
 \cite{Serfozo} \cite{Guo}. 
  The rewards, transition probabilities, and discount factor of  DTMDP($g$) equivalent to CTMDP($g$) are 
  given by, 
  \begin{equation}\label{pl1-eqdata}
   \left.
   \begin{aligned}
    \bar{r}^1(s,a^1)&=\frac{r^1(s,a^1,g)}{||\mu||+\alpha},\ \forall \ s\in S, a^1\in A^1(s),\\ 
  p^1(s'|s,a^1)&=\frac{\mu(s',s,a^1,g)}{||\mu||}+\delta(s,s'), \ \forall \ s,s'\in S, a^1\in A^1(s), \\
 \beta&=\frac{||\mu||}{\alpha+||\mu||},
   \end{aligned}
\right\}
  \end{equation}
  where $\delta(\cdot)$ is a Kronecker delta.
 Similarly, for a fixed stationary strategy $f$ of player 1, player 2 faces a 
 CTMDP($f$). The immediate rewards and transition rates  of CTMDP($f$) are respectively given by 
 $r^2(s,f,a^2)=\sum_{a^1\in A^1(s)} r^2(s,a^1,a^2)f(s,a^1)$ and 
 $\mu(s',s,f,a^2)=\sum_{a^1\in A^1(s)} \mu(s',s,a^1,a^2) f(s,a^1)$
 for all $s,s'\in S$, $a^2\in A^2(s)$. 
 The rewards, transition probabilities, and discount factor of  DTMDP($f$) equivalent to CTMDP($f$) are 
  given by,
 \begin{equation}\label{pl2-eqdata}
 \left. 
 \begin{aligned}
   \bar{r}^2(s,a^2)&=\frac{r^2(s,f,a^2)}{||\mu||+\alpha}, \ \forall \ s\in S, a^1\in A^1(s),\\
 p^2(s'|s,a^2)&=\frac{\mu(s',s,f,a^2)}{||\mu||}+\delta(s,s'), \ \forall \ s,s'\in S, a^2\in A^2(s), \\ 
 \beta&=\frac{||\mu||}{\alpha+||\mu||}.
 \end{aligned}
\right\}
 \end{equation}
The continuous time stochastic game is defined using the transition rates. 
Let $Q(f,g)$  denote the transition  rate matrix 
induced by a stationary strategy pair $(f,g)$, where
$Q(f,g)=[\mu(s',s,f,g)]_{ss'}.$

 \subsection{Blackwell-Nash equilibrium in continuous time stochastic games}
 A  BNE for  continuous time stochastic games can be defined similar to discrete time stochastic games as follows:
 \begin{definition}
 A strategy pair $(f^*,g^*)$ is said to be a BNE of a continuous time stochastic game 
 if there exists an $\alpha_0>0$ such that $(f^*,g^*)$ is an 
 $\alpha$-discounted Nash equilibrium for every $\alpha\in (0,\alpha_0]$.
\end{definition}
 We provide the results on BNE for continuous time stochastic games along similar lines. 
We first show that a stationary BNE in general continuous time stochastic  games need not alway exist.
We give  an example of a single controller continuous time stochastic game where stationary BNE does not exist.
This example shows that BNE need not always exist even for single controller games which is a special 
class of general stochastic games.
Then, we show the existence of a stationary deterministic BNE 
for continuous time SC-AR stochastic games. 
Finally, for general continuous time stochastic games we give two different sets of conditions and show that each set of conditions
together are sufficient 
for a Nash  equilibrium to be a BNE. 

 \subsubsection{A counter example}
 We give an example which does not have any stationary BNE. 
\begin{example}\label{Cont-SC-AR-Ex1}
We consider a continuous time stochastic game with 2 states and both players having two actions 
at state 1 and only one action at state 2, i.e.,
$S=\{1,2\}$, $A^1(1)=A^2(1)=\{1,2\}$, $A^1(2)=A^2(2)=\{1\}$. The rewards of both the players and transition rates for 
different combination of states and actions are summarized in the Table \ref{Tb2c}. The upper half of 
each box of table represents transition rates and lower half represents immediate rewards.

\begin{table}[ht]
\centering 
\caption{Immediate rewards and Transition Rates} 
\label{Tb2c}
\subfigure[$s=1$]{\resizebox{0.6\linewidth}{.9cm}
{
\begin{tabular}{|r|r|}
\hline
\backslashbox{(4, 9)}{(0, 0)} & \backslashbox{(6, 3)} {(-1, 1)}\\
\hline
\backslashbox{(5, 4)}{(0, 0)} &\backslashbox{(4, 5)}{(-1, 1)}\\
\hline
\end{tabular}
}}
\subfigure[$s=2$]{\resizebox{0.3\linewidth}{.5cm}
{
\begin{tabular}{|r|}
\hline
\backslashbox{(6, 7)}{(1, -1)} \\
\hline
\end{tabular}.
}}
\end{table} 
\end{example}
The Example \ref{Cont-SC-AR-Ex1} can be viewed as a continuous time version of \mbox{Example \ref{example1}}.
We show that the Example \ref{Cont-SC-AR-Ex1} does not  possess a stationary BNE. 
Let \mbox{$(f,g)=((p,1-p),(q,1-q))$}, for some $0\leq p,q\leq 1$, be an arbitrary stationary strategy pair.
For fixed $g$, player 1 faces 
a CTMDP($g$). From the data of the game $||\mu||=1$. The CTMDP($g$)  
is equivalent to the DTMDP($g$) defined by \eqref{pl1-eqdata}. 
The transition probabilities of DTMDP($g$)
do not depend on the actions of player 1 because the 
transition rates do not depend on the actions of player~1. So,  $f^*$ is an optimal policy 
of player 1 for DTMDP($g$) if and only if for each $s\in S$,
\begin{equation}\label{sing_beq1}
 \bar{r}^1(s,f^*)=\max_{f\in F_S} \bar{r}^1(s,f)=\frac{1}{1+\alpha}\max_{a^1\in A^1(s)}[R^1(s)g(s)]_{a^1}.
\end{equation}
We need to determine $f^*$ only at state $s=1$. We have, 
$$R^1(1)g(1)=[6-2q,4+q]^T.$$
Let $f_1=(1,0)$ and $f_2=(0,1)$ be two stationary deterministic strategies of player 1.
From \eqref{sing_beq1}, 

\begin{align}\label{sing_beq2}
f^*= 
\begin{cases} f_1 &  \text{if} \ q<\frac{2}{3} 
\\
f_2 &\text{if} \ q>\frac{2}{3}\\
\big\{(p,1-p):  0\leq p\leq 1\big\} & \text{if} \ q=\frac{2}{3}.
\end{cases}
\end{align}
Equation \eqref{sing_beq2} gives the optimal policy of player 1 for DTMDP($g$)  for all \mbox{$\beta\in [0,1)$}.
 Therefore, $f^*$ gives an optimal policy of 
CTMDP($g$) for all $\alpha>0$. That is,
$f^*$ gives the best response of player 1 for all $\alpha>0$ for a fixed 
strategy $g=(q,1-q)$ of player~2.

Similarly, for a fixed $f=(p,1-p)$, player 2 faces a  CTMDP($f$). 
The  equivalent DTMDP($f$) is defined by \eqref{pl2-eqdata}. 
Let $g_1=(1,0)$ and $g_2=(0,1)$ be two stationary 
deterministic strategies of player 2. By using the data given in Example \ref{Cont-SC-AR-Ex1}, the value vector of player 2 
for DTMDP($f$)  is given below:
\begin{equation}\label{sing_beq3}
 u_{\beta}^2(f,g_1)=[I-\beta P(g_1)]^{-1}\bar{r}^2(g_1)= \beta\left[\frac{4+5p}{1-\beta},\frac{(4+5p)\beta}{1-\beta}+7\right]^T.
\end{equation}
\begin{equation}\label{sing_beq4}
 u_{\beta}^2(f,g_2)=[I-\beta P(g_2)]^{-1}\bar{r}^2(g_2)=\beta\left[\frac{5-2p+7\beta}{1-\beta^2},
 \frac{(5-2p)\beta+7}{1-\beta^2}\right]^T.
\end{equation}
By using \eqref{sing_beq3} and \eqref{sing_beq4} we have,
\begin{equation}\label{sing_beq5}
 u_{\beta}^2(f,g_1)-u_{\beta}^2(f,g_2)=\beta
 \left[\frac{p(7+5\beta)-(3\beta+1)}{1-\beta^2}, \frac{\beta(p(7+5\beta)-(3\beta+1))}{1-\beta^2}\right]^T.
\end{equation}
By substituting $\beta=\frac{1}{1+\alpha}$ in \eqref{sing_beq5}, the  difference in the value vector of CTMDP($f$)
is given by 
\begin{equation}\label{sing_beq5'}
  v_\alpha^2(f,g_1)-v_\alpha^2(f,g_2)=\left[\frac{p(12+7\alpha)-(4+\alpha)}{\alpha(\alpha+2)}, 
  \frac{p(12+7\alpha)-(4+\alpha)}{\alpha(\alpha+2)(1+\alpha)}\right]^T.
\end{equation}
From \eqref{sing_beq5'} the best response $g^*$ of player 2 against a fixed strategy 
$f=(p,1-p)$ of player 1 for a given discount rate $\alpha$
is given by \eqref{sing_beq6}
 
\begin{align}\label{sing_beq6}
g^*= 
\begin{cases} g_1 &  \text{if} \ p>\frac{4+\alpha}{12+7\alpha} 
\\
g_2 &\text{if} \ p<\frac{4+\alpha}{12+7\alpha} \\
\big\{(q,1-q):  0\leq q\leq 1\big\} & \text{if} \ p=\frac{4+\alpha}{12+7\alpha} .
\end{cases}
\end{align}
 From \eqref{sing_beq2}
and \eqref{sing_beq6} it is easy to see that for a discount rate 
$\alpha$, a strategy pair 
$(f_\alpha^*,g_\alpha^*)=
\left(\left(\frac{4+\alpha}{12+7\alpha},\frac{8+6\alpha}{12+7\alpha}\right),\left(\frac{2}{3},\frac{1}{3}\right)\right)$ 
is such that $f_\alpha^*$ and $g_\alpha^*$ are 
best responses of each other, i.e., it is a Nash equilibrium.  The uniqueness of $(f_\alpha^*,g_\alpha^*)$ follows from the 
similar arguments used in Example \ref{example1}.
Since, $f_\alpha^*$ is an invertible function of $\alpha$, then the Nash equilibrium $(f_\alpha^*,g_\alpha^*)$ 
varies with discount rate $\alpha$. This implies that Example~\ref{Cont-SC-AR-Ex1} will not have a stationary BNE. 

From Example \ref{Cont-SC-AR-Ex1} it is clear that in general 
a continuous time stochastic game need not admit a stationary BNE.
In fact Example \ref{Cont-SC-AR-Ex1} belongs 
to the class of single controller games. So, even for the class of single controller games there is 
no guarantee that a stationary BNE will exist.
 Next, we describe SC-AR stochastic games which is a special class of
single controller games. 
Similar to discrete  case  we show that there always 
exists a stationary deterministic BNE. 

 \subsubsection{Single Controller Additive Reward Games}
   A continuous time SC-AR stochastic game  is characterized by the following assumptions:
 \begin{itemize}
  \item [(a)]$\mu(s',s,a^1,a^2)=\mu(s',s,a^2)$ for all $s',s\in S$, $a^1\in A^1(s)$, $a^2\in A^2(s)$, i.e., the transition rates only depend 
  on the actions of player 2. 
  \item [(b)] $r^1(s,a^1,a^2)=r_1^1(s,a^1)+r_2^1(s,a^2)$, for all $s\in S$, $a^1\in A^1(s)$, $a^2\in A^2(s)$.
 \end{itemize}
 
 \begin{theorem}
  Every continuous time SC-AR stochastic game possesses a stationary deterministic 
  BNE.
 \end{theorem}
 
 \begin{proof}
  For each $s\in S$ select an action $a_s^{1*}\in A^1(s)$ such that 
  $a_s^{1*}\in \underset{a^1\in A^1(s)}{\operatorname{argmax}}\{r_1^1(s,a^1)\}$. 
  Define $f^*\in F_S$ by 
  \begin{equation}\label{SC-AR_pol}
 f^*(s,a^1)=
\begin{cases} 1 & \text{if $a^1=a_s^{1*}$,}
\\
0 &\text{otherwise}
\end{cases}
\end{equation}
for each $s\in S$. For above strategy $f^*$ of player 1, player 2 faces a CTMDP($f^*$).
The equivalent DTMDP($f^*$) is defined by \eqref{pl2-eqdata}.
For DTMDP($f^*$) there always exists a stationary deterministic strategy $g^*$ which is  Blackwell optimal \cite{David}. 
Then, there exists  a discount factor $\beta_0$ such that $g^*$ is an optimal strategy 
for all $\beta\in [\beta_0,1)$, i.e., for all $s\in S$,
\begin{equation}\label{BO-eq}
 u_{\beta}^2(s,f^*,g^*)\geq u_{\beta}^2(s,f^*,g), \ \forall \ g\in G_S , \beta \in [\beta_0,1),
\end{equation}
where $u_{\beta}^2(s,f^*,g)$ is the expected discounted reward of the DTMDP($f^*$) 
for a given initial state $s$ and strategy $g$. 
From \cite{Serfozo}, $v_\alpha^2(s,f^*,g)=u_{\beta}^2(s,f^*,g)$, for all $s\in S$, where relationship between $\alpha$ and
$\beta$ is given by \eqref{pl2-eqdata}. Then, for discount factor $\beta_0$ we have 
discount rate $\alpha_0= ||\mu||\big(\frac{1-\beta_0}{\beta_0}\big)$. Therefore, from \eqref{BO-eq} we have for all $s\in S$
\begin{equation}\label{BNE_eq1}
 v_\alpha^2(s,f^*,g^*)\geq v_\alpha^2(s,f^*,g), \ \forall \ g\in G_S, \alpha\in(0,\alpha_0].
\end{equation}
From \eqref{SC-AR_pol}, we have 
\begin{equation}\label{SC-AR_e1}
 r^1(f^*,g^*)\geq r^1(f,g^*), \;\; \forall\;   f\in F_S.
\end{equation}
Because the transitions rates do not depend on the strategies $f\in F_S$, therefore, we have for all $s\in S$, 
\begin{equation}\label{BNE_eq2}
 v_\alpha^1(s,f^*,g^*)\geq v_\alpha^1(s, f, g^*), \ \forall \ f\in F_S, \alpha>0.
\end{equation}
From \eqref{BNE_eq1} and \eqref{BNE_eq2}, $(f^*,g^*)$ is a BNE.
\end{proof}

\subsubsection{Sufficient conditions for BNE in general stochastic games}
 We consider a two player general continuous time stochastic game with discounted payoff criterion. 
 We give two disjoint  sets of conditions where each set of conditions together are sufficient for a stationary Nash equilibrium 
 to be a BNE.\\

\noindent \textbf{First set of sufficient conditions:}\\
\noindent \textbf{M1}. $(f^*,g^*)$ is a pure strategy Nash equilibrium.\\

 \noindent \textbf{M2}. $$Q(f^*,g^*)=\begin{matrix}||\mu||
\begin{pmatrix}
  p_1-1 & p_2 &  \cdots & p_{|S|} \\
  p_1 & p_2-1 &  \cdots & p_{|S|} \\
  \vdots&  \vdots&  &\vdots\\
  p_1 & p_2 &  \cdots & p_{|S|}-1
 \end{pmatrix}
\end{matrix},
$$
where $p_s\geq 0, \ \forall \ s\in S, \ \sum_{s\in S}p_s=1$. \\

\noindent\textbf{M3}. $\left\{ \begin{array}{l}\displaystyle
\sum_{s'\in S} p_{s'} \ r^1(s',a_{s'}^1,a_{s'}^2)\geq \sum_{s'\in S} \left(\frac{\mu(s',s,a^1,a_s^2)}{||\mu||}+\delta(s,s')\right) 
r^1(s',a_{s'}^1,a_{s'}^2),\\
\hspace{7.7cm}\forall \ s\in S, a^1\in A^1(s),\\ \vspace{.2cm}\\ \displaystyle
\sum_{s'\in S} p_{s'} \ r^2(s',a_{s'}^1,a_{s'}^2)\geq \sum_{s'\in S} \left(\frac{\mu(s',s,a_s^1,a^2)}{||\mu||}+\delta(s,s')\right) 
r^2(s',a_{s'}^1,a_{s'}^2),\\
\hspace{7.7cm} \forall\; s\in S, a^2\in A^2(s).
\end{array}\right.$
\smallskip


\begin{theorem}\label{Blackwell_disc_cont_thm1}
 If $(f^*,g^*)$ is a stationary Nash equilibrium of a discounted continuous time stochastic game at some discount rate $\hat{\alpha}>0$
and satisfies the conditions \textbf{M1}, \textbf{M2} and \textbf{M3}, then it will be a BNE. 
\end{theorem}
 
 \begin{proof}
Let $(f^*,g^*)$ be a stationary Nash equilibrium of a continuous time discounted stochastic game at some discount rate $\hat{\alpha}>0$.
Then, $f^*$ is an optimal policy of  CTMDP($g^*$) at discount rate $\hat{\alpha}$.
Therefore,  $f^*$ is an optimal policy of the equivalent DTMDP($g^*$), defined by \eqref{pl1-eqdata},
at $\hat{\beta}=\frac{||\mu||}{\alpha+||\mu||}$ \cite{Serfozo}.
We are interested in the range of $\beta$ for which $f^*$ is an optimal policy of DTMDP($g^*$).
That is, the range of $\beta$ for which the optimality equations for the DTMDP($g^*$) given below are satisfied by $f^*$,
 \begin{equation}\label{DT_MDP_opteq2}
 u^{1*}(s)=\bar{r}^1(s,a_s^1)+\beta \sum_{s'\in S} p^1(s'|s,a_s^1) u^{1*}(s'), \ \forall \ s\in S,
\end{equation}
and 
\begin{equation}\label{DT_MDP_opteq3}
 u^{1*}(s)\ge\bar{r}^1(s,a^1)+\beta \sum_{s'\in S} p^1(s'|s,a^1) u^{1*}(s'), \ \forall \ s\in S, a^1\in A^1(s), a^1\neq a_s^1,
\end{equation}
where $u^{1*}$ is the value vector of player 1 for $f^*$.
That is,
$$u^{1*}=u_\beta^1(f^*)=(I-\beta P^1(f^*))^{-1}\bar{r}^1(f^*),$$
where transitions probability matrix induced by $f^*$ for DTMDP($g^*$) is given by,
\[
 P^1(f^*)=\frac{Q(f^*,g^*)}{||\mu||}+I=\begin{matrix}
\begin{pmatrix}
  p_1 & p_2 &  \cdots & p_{|S|} \\
  p_1 & p_2 &  \cdots & p_{|S|} \\
  \vdots&  \vdots&  &\vdots\\
  p_1 & p_2 &  \cdots & p_{|S|}
 \end{pmatrix}
\end{matrix}.
\]
 From direct calculation we have,
\[
 u^1(f^*)=\bar{r}^1(f^*)+\frac{\beta}{1-\beta}\sum_{s\in S}p_s\bar{r}^1(s,a_s^1)\textbf{1}_{|S|}.
\]
It is easy to see that  \eqref{DT_MDP_opteq2} holds. 
Denote,
\begin{equation}\label{DT_MDP_opteq4}
 \theta_{s,a^1}^1=\bar{r}^1(s,a^1)+\beta \sum_{s'\in S} p^1(s'|s,a^1) u^{1*}(s')-u_\beta^{1*}(s).
\end{equation}
for all $s\in S$, $a^1\in A^1(s)$, $a^1\neq a_s^1$.
By substituting the value of $u^{1*}$ in
\eqref{DT_MDP_opteq4}, we have 

\begin{equation}\label{DT_MDP_opteq5}
\theta_{s,a^1}^1= \big(\bar{r}^1(s,a^1)-\bar{r}^1(s,a_s^1)\big)
-\beta\left(\sum_{s'\in S}p_{s'}\bar{r}^1(s',a_{s'}^1)-\sum_{s'\in S}p^1(s'|s,a^1)
\bar{r}^1(s',a_{s'}^1)\right),
\end{equation}
for all   $s\in S$, $a^1\in A^1(s)$, $a^1\neq a_s^1$.
If $\sum_{s'\in S}p_{s'}\bar{r}^1(s',a_{s'}^1)=\sum_{s'\in S}p(s'|s,a^1)\bar{r}^1(s',a_{s'}^1)$ 
for some $s\in S$, $a^1\in A^1(s)$, $a^1\neq a_s^1$,
then, \eqref{DT_MDP_opteq5} is independent of $\beta$.  Therefore,
$\theta_{s,a^1}^1\leq 0, \forall \ \beta$ because it holds  for $\hat{\beta}$. 
 In other cases from {\bf M3}, we have, 
$\theta_{s,a^1}^1\leq 0, \forall \ \beta \geq \beta_{s,a^1}^1$, 
where 
\begin{align}\label{eee1}
 \beta_{s,a^1}^1&=\frac{\bar{r}^1(s,a^1)-\bar{r}^1(s,a_s^1)}{\sum_{s'\in S}p_{s'}\bar{r}^1(s',a_{s'}^1)-\sum_{s'\in S}p^1(s'|s,a^1)
\bar{r}^1(s',a_{s'}^1)}\nonumber\\
&=\frac{r^1(s,a^1,a_s^2)-r^1(s,a_s^1,a_s^2)}{\sum_{s'\in S}p_{s'}r^1(s',a_{s'}^1,a_{s'}^2)
-\sum_{s'\in S}\left(\frac{\mu(s',s,a^1,a_s^2)}{||\mu||}+\delta(s,s')\right)
r^1(s',a_{s'}^1,a_{s'}^2)}.
\end{align}
It is clear that $\beta_{s,a^1}^1\leq\hat{\beta}<1$, because $f^*$ is an optimal policy at $\hat{\beta}$. 
Define,
\begin{equation}\label{eee2}
 \beta_0^1=\max_{s\in S,a^1\in A^1(s),a^1\neq a_s^1}\{0,\beta_{s,a^1}^1\},
\end{equation}
whenever $\beta_{s,a^1}^1$ is well defined. We include ``0'' in \eqref{eee2} because 
$\beta_{s,a^1}^1$ defined by \eqref{eee1} can be negative. Now, $f^*$ is an optimal policy of the DTMDP($g^*$)
for all $\beta\in [\beta_0^1,1)$. 
From \cite{Serfozo}, $f^*$ is an optimal policy of the  CTMDP($g^*$) for all $\alpha\in (0,\alpha_0^1]$, where,
$$\alpha_0^1=\frac{(1-\beta_0^1)||\mu||}{\beta_0^1}.$$
Therefore, $f^*$ is a best response of $g^*$ for all $\alpha\in (0,\alpha_0^1]$.

For fixed $f^*$, player 2 faces a CTMDP($f^*$) whose optimal policy is $g^*$ at discount rate $\hat{\alpha}$.
Therefore, $g^*$ is an optimal policy of player 2 for  the equivalent 
 DTMDP($f^*$), defined by \eqref{pl2-eqdata},
 at $\hat{\beta} = \frac{||\mu||}{\hat{\alpha}+||\mu||}$.
We are interested in finding the range of $\beta$ for which the optimality equations for DTMDP($f^*$) given below
are satisfied at $g^*$.
\begin{equation}\label{DT_MDP_opteq7}
 u^{2*}(s)=\bar{r}^2(s,a_s^2)+\beta \sum_{s'\in S} p^2(s'|s,a_s^2) u^{2*}(s'), \ \forall \ s\in S.
\end{equation}
\begin{equation}\label{DT_MDP_opteq8}
 u^{2*}(s)\ge\bar{r}^2(s,a^2)+\beta \sum_{s'\in S} p^2(s'|s,a^2) u^{2*}(s'), \ \forall \ s\in S, a^2\in A^2(s), a^2\neq a_s^2,
\end{equation}
where $u^{2*}$ is the value vector of player 2 at $g^*$. The transition probability matrix induced by $g^*$ for 
DTMDP($f^*$) is given by,
\[
 P^2(g^*)=\frac{Q(f^*,g^*)}{||\mu||}+I.
\]
As similar to previous case,
$$u^{2*}=[I-\beta P^2(g^*)]^{-1}\bar{r}^2(g^*)=\bar{r}^2(g^*)
+\frac{\beta}{1-\beta}\sum_{s\in S} p_s \bar{r}^2(s,a_s^2) {\bf 1}_{|S|}.$$ 
It is clear that  \eqref{DT_MDP_opteq7} holds. Denote,
 \begin{equation}\label{DT_MDP_opteq9}
  \theta_{s,a^2}^2=\bar{r}^2(s,a^2)+\beta \sum_{s'\in S} p^2(s'|s,a^2) u^{2*}(s')-u^{2*}(s),
  \ \forall \ s\in S, a^2\in A^2(s), a^2\neq a_s^2.
 \end{equation}
 By substituting the value of $u^{2*}$ in \eqref{DT_MDP_opteq9} we have 
 \begin{equation}\label{DT_MDP_opteq10}
  \theta_{s,a^2}^2=\big(\bar{r}^2(s,a^2)-\bar{r}^2(s,a_s^2)\big)-\beta\left(\sum_{s'\in S} p_{s'} \bar{r}^2(s',a_{s'}^2)
  -\sum_{s'\in S} p^2(s'|s,a^2) \bar{r}^2(s',a_{s'}^2)\right),  
 \end{equation}
 for all  $s\in S$, $a^2\in A^2(s)$, $a^2\neq a_s^2$. If $\sum_{s'\in S} p_{s'} \bar{r}^2(s',a_{s'}^2)
  =\sum_{s'\in S} p(s'|s,a^2) \bar{r}^2(s',a_{s'}^2)$ for some $s\in S$, $a^2\in A^2(s)$, $a^2\neq a_s^2$,  then \eqref{DT_MDP_opteq10}
  is independent of $\beta$ and $\theta_{s,a^2}^2\leq 0$ for all $\beta$. In other cases from {\bf M3} we have, 
  $\theta_{s,a^2}^2\leq 0$ for all $\beta\geq \beta_{s,a^2}^2$, where 
  
  \begin{align*}
  \beta_{s,a^2}^2&=\frac{\bar{r}^2(s,a^2)-\bar{r}^2(s,a_s^2)}{\sum_{s'\in S} p_{s'} \bar{r}^2(s',a_{s'}^2)
  -\sum_{s'\in S} p^2(s'|s,a^2) \bar{r}^2(s',a_{s'}^2)}\nonumber\\
  &=\frac{r^2(s,a_s^1,a^2)-r^2(s,a_s^1,a_s^2)}{\sum_{s'\in S} p_{s'} r^2(s',a_{s'}^1,a_{s'}^2)
  -\sum_{s'\in S} \left(\frac{\mu(s',s,a_s^1,a^2)}{||\mu||}+\delta(s,s')\right) r^2(s',a_{s'}^1,a_{s'}^2)}. 
  \end{align*}
It is clear that $\beta_{s,a^2}^2\leq \hat{\beta}<1$, because $g^*$ is an optimal policy at $\hat{\beta}$. Define 
\begin{equation*}
 \beta_0^2=\max_{s\in S,a^2\in A^2(s),a^2\neq a_s^2}\{0,\beta_{s,a^2}^2\},
\end{equation*}
whenever $\beta_{s,a^2}^2$ is well defined. This implies that $g^*$ is an optimal policy of the DTMDP($f^*$) 
for all $\beta\in [\beta_0^2,1)$.
From \cite{Serfozo}, $g^*$ is an optimal policy of the CTMDP($f^*$) for all $\alpha\in (0,\alpha_0^2]$, where
$$\alpha_0^2=\frac{(1-\beta_0^2)||\mu||}{\beta_0^2}.$$
 That is, $g^*$ is a best response of $f^*$ for all $\alpha\in (0,\alpha_0^2]$. Define,
\begin{equation}\label{Blackwell_bound1}
 \alpha_0=\min\{\alpha_0^1,\alpha_0^2\}.
\end{equation}
We can say that $f^*$ and $g^*$ are best response of each other for all $\alpha\in (0,\alpha_0]$. So, $(f^*,g^*)$ is 
a Nash equilibrium of a continuous time $\alpha$-discounted stochastic game for all $\alpha\in (0,\alpha_0]$, i.e., it 
is a BNE. 
\end{proof}

Now, we give an example of a continuous time stochastic game 
that possess a Nash equilibrium which satisfies \textbf{M1}, \textbf{M2} and \textbf{M3}.

\begin{example}\label{ex3}
 We consider a 2 states  continuous time stochastic game where  both the players 
 have two actions at state 1 and  only one action at state 2 , i.e.,
$S=\{1,2\}$, $A^1(1)= A^2(1)= \{1,2\}$, $A^1(2)=A^2(2)=\{1\}$. 
The rewards of both the players and the transition rates for different combinations 
of states and actions are summarized in the Table \ref{Tb3c}. 

\begin{table}[ht]
\centering 
\caption{Immediate rewards and Transition Rates} 
\label{Tb3c}
\subfigure[$s=1$]{\resizebox{0.6\linewidth}{.9cm}
{
\begin{tabular}{|r|r|}
\hline
\backslashbox{(5,3)}{(0,0)} & \backslashbox{(2,3)} {(-1,1)}\\
\hline
\backslashbox{(3,4)}{(-1,1)} &\backslashbox{(4,2)}{(0,0)}\\
\hline
\end{tabular}
}}
\subfigure[$s=2$]{\resizebox{0.3\linewidth}{.5cm}
{
\begin{tabular}{|r|}
\hline
\backslashbox{(5,4)}{(0, 0)} \\
\hline
\end{tabular}.
}}
\end{table} 
\end{example}

\noindent We show that $(f^*,g^*)=((1,0),(0,1))$ is a Blackwell Nash equilibrium of the continuous time stochastic game  given
in above example.
We first show that $(f^*,g^*)$ is a Nash equilibrium at $\alpha=0.5$. From the data of the game $||\mu||=1$.
Fix $g^*=(0,1)$, then player 1 faces a CTMDP($g^*$). 
The optimal policy of CTMDP($g^*$) at discount rate $\alpha=0.5$ 
can be computed by solving an equivalent DTMDP($g^*$), defined by \eqref{pl1-eqdata}, 
at discount factor $\beta=\frac{1}{1+\alpha}=~0.67$.
 It is known that the optimal policy of 
a DTMDP exists among the class of stationary deterministic policies. 
Let $f_1=(1,0)$ and $f_2=(0,1)$ be two stationary deterministic policies 
 for player 1. From the above data, the transition probability matrices induced by $f_1$ and $f_2$ for DTMDP($g^*$) are given by, 

$$P^1(f_1)=
\begin{matrix}
 \begin{pmatrix}
0 ~ & 1\\
0 ~ & 1
\end{pmatrix}
\end{matrix}, 
\hspace{.25cm}
 P^1(f_2)=  \begin{matrix}
 \begin{pmatrix}
1 ~ & 0\\
0 ~ & 1
       \end{pmatrix}
\end{matrix}.
$$
We have, 
\begin{equation}\label{e5}
 u_{0.67}^1(f_1)=[I-0.67 P^1(f_1)]^{-1}\bar{r}^1(f_1)=(8,10).
\end{equation}
\begin{equation}\label{e6}
  u_{0.67}^1(f_2)=[I-0.67 P^1(f_2)]^{-1}\bar{r}^1(f_2)=(8,10).
\end{equation}
From \eqref{e5} and \eqref{e6}, $f_1=(1,0)=f^*$ and $f_2$ both are optimal policy of DTMDP($g^*$) at $\beta=0.67$. 
This implies $f^*$ is an
optimal policy of CTMDP($g^*$) at $\alpha =0.5$, i.e., $f^*$ is a best response of $g^*$. Fix $f^*=(1,0)$, then player 2 faces a 
CTMDP($f^*$).   
The optimal policy of CTMDP($f^*$)
can be computed by solving an equivalent DTMDP($f^*$) defined by \eqref{pl2-eqdata}.
Let $g_1=(1,0)$ and $g_2=(0,1)$ be  two stationary deterministic policies
 of player 2. The transition probability matrices induced by $g_1$ and $g_2$ are given by,

$$P^2(g_1)=
\begin{matrix}
 \begin{pmatrix}
1 ~ & 0\\
0 ~ & 1
\end{pmatrix}
\end{matrix}, 
\hspace{.25cm}
 P^2(g_2)=  \begin{matrix}
 \begin{pmatrix}
0 ~ & 1\\
0 ~ & 1
       \end{pmatrix}
\end{matrix}.
$$
We have, 
\begin{equation}\label{e7}
 u_{0.67}^2(g_1)=[I-0.67 P^2(g_1)]^{-1}\bar{r}^2(g_1)=(6,8).
\end{equation}
\begin{equation}\label{e8}
  u_{0.67}^2(g_2)=[I-0.67 P^2(g_2)]^{-1}\bar{r}^2(g_2)=(7.33,8).
\end{equation}
From \eqref{e7} and \eqref{e8} $g_2=(0,1)=g^*$ is an optimal policy of DTMDP($f^*$) at $\beta=0.67$. This implies $g^*$ is an
optimal policy of CTMDP($f^*$) at $\alpha =0.5$, i.e., $g^*$ is a best response of $f^*$. Hence $(f^*,g^*)$ is a Nash equilibrium 
 at $\alpha=0.5$. It easy to check that 
$(f^*,g^*)$ satisfies conditions \textbf{M1}, \textbf{M2} and \textbf{M3}. Hence, from Theorem \ref{Blackwell_disc_cont_thm1} it is a 
BNE. From \eqref{Blackwell_bound1}, $\alpha_0=0.5$, so $(f^*,g^*)$ is a Nash equilibrium for all
$\alpha\in (0,0.5]$. \\

\noindent \textbf{Second set  of sufficient conditions:}\label{second-set}\\
\textbf{N1.} $(f^*,g^*)$ is a pure strategy Nash equilibrium\\

\noindent \textbf{N2.} $$Q(f^*,g^*)=\begin{matrix}
\begin{pmatrix}
  0 ~ & 0 &  \cdots & 0 \\
   0 ~ & 0 &  \cdots & 0 \\
  \vdots&  \vdots&  &\vdots\\
  0 ~ & 0 &  \cdots & 0
 \end{pmatrix}
\end{matrix},
$$
i.e., all the states of Markov chain induced by $(f^*,g^*)$  are absorbing.
\smallskip

\noindent \textbf{N3.}
$\left\{ \indent \begin{array}{l}\displaystyle
 r^1(s,a^1,a_{s}^2)\geq \sum_{s'\in S}  \left(\frac{\mu(s',s,a^1,a_s^2)}{||\mu||}+\delta(s,s')\right) r^1(s',a_{s'}^1,a_{s'}^2),\\
\hspace{7cm}\; \forall\; s\in S, a^1\in A^1(s),\\ \vspace{.2cm}\\\displaystyle
 r^2(s,a_{s}^1,a^2)\geq \sum_{s'\in S} \left(\frac{\mu(s',s,a_s^1,a^2)}{||\mu||}+\delta(s,s')\right) r^2(s',a_{s'}^1,a_{s'}^2),\\
\hspace{7cm}\; \forall\; s\in S, a^2\in A^2(s).
\end{array}\right.$
\smallskip


\begin{theorem}\label{Blackwell_disc_cont_thm2}
 If $(f^*,g^*)$ is a stationary Nash equilibrium of a discounted continuous time stochastic game at some discount rate $\hat{\alpha}>0$
and satisfies the conditions \textbf{N1}, \textbf{N2} and \textbf{N3}, then it will be a BNE. 
\end{theorem}

\begin{proof}
 The proof follows using the similar arguments as in Theorem \ref{Blackwell_disc_cont_thm1}. The required discount rate $\alpha_0$ is 
 given by 
 \begin{equation}\label{Blackwell-bound}
  \alpha_0=\min\{\alpha_0^1,\alpha_0^2\},
 \end{equation}
 where $\alpha_0^i=\frac{(1-\beta_0^i)||\mu||}{\beta_0^i}$, $i=1,2$. The bounds $\beta_0^i$, $i=1,2$, can be calculated from \eqref{bounds}
 \begin{equation}\label{bounds}
  \beta_0^i=\max_{s\in S,a^i\in A^i(s),a^i\neq a_s^i}\{0,\beta_{s,a^i}^i\}, \ i=1,2,
 \end{equation}
where the bounds $\beta_{s,a^1}^1$ and $\beta_{s,a^2}^2$, whenever well defined, are given by 
\begin{equation}
\beta_{s,a^1}^1=\frac{r^1(s,a^1,a_s^2)-r^1(s,a_s^1,a_s^2)}{r^1(s,a^1,a_s^2)
-\sum_{s'\in S}\left(\frac{\mu(s',s,a^1,a_s^2)}{||\mu||}+\delta(s,s')\right) r^1(s',a_{s'}^1,a_{s'}^2)},
\end{equation}

\begin{equation}
 \beta_{s,a^2}^2=\frac{r^2(s,a_s^1,a^2)-r^2(s,a_s^1,a_s^2)}{r^2(s,a_s^1,a^2)-\sum_{s'\in S}
 \left(\frac{\mu(s',s,a_s^1,a^2)}{||\mu||}+\delta(s,s')\right)r^2(s',a_{s'}^1,a_{s'}^2)}, 
\end{equation}
for all $s\in S, a^1\in A^1(s), a^1\neq a_s^1, a^2\in A^2(s), a^2\neq a_s^2$.
\end{proof}

Now, we give an example and show that  there exists 
a stationary Nash equilibrium which satisfies  conditions \textbf{N1}, \textbf{N2} and \textbf{N3}.

\begin{example}\label{ex2}
 We consider a 2 states  continuous time stochastic game where  both the players 
 have two actions at state 1 and  only one action at state 2 , i.e.,
$S=\{1,2\}$, $A^1(1)= A^2(1)= \{1,2\}$, $A^1(2)=A^2(2)=\{1\}$. 
The rewards of both the players and the transition rates for different combinations 
of states and actions are summarized in the Table \ref{Tb4c}. 

\begin{table}[ht]
\centering 
\caption{Immediate rewards and Transition Rates} 
\label{Tb4c}
\subfigure[$s=1$]{\resizebox{0.6\linewidth}{.9cm}
{
\begin{tabular}{|r|r|}
\hline
\backslashbox{(4,4.4)}{(0,0)} & \backslashbox{(4,5)} {(-1,1)}\\
\hline
\backslashbox{(5,6)}{(-1,1)} &\backslashbox{(3,2)}{(0,0)}\\
\hline
\end{tabular}
}}
\subfigure[$s=2$]{\resizebox{0.3\linewidth}{.5cm}
{
\begin{tabular}{|r|}
\hline
\backslashbox{(3,4)}{(0, 0)} \\
\hline
\end{tabular}.
}}
\end{table} 
\end{example}
\noindent The Example \ref{ex2} can be viewed as a continuous time version of Example \ref{Ex-sec-set}.
We show  that $(f^*,g^*)=((1,0),(1,0))$ is a BNE.  We first show that $(f^*,g^*)$ is a Nash equilibrium at 
$\alpha=\frac{2}{3}$.  From data of the game $||\mu||=1$. 
Fix  $g^*=(1,0)$, then first player faces a CTMDP($g^*$). 
The optimal policy of CTMDP($g^*$) at $\alpha=\frac{2}{3}$ can be computed by solving 
the equivalent DTMDP($g^*$),  defined by \eqref{pl1-eqdata}, at $\beta=\frac{1}{1+\alpha}=0.6$.
Let $f_1=(1,0)$ and $f_2=(0,1)$ be two stationary deterministic policies. 
Using the above data, the transition probability matrices induced by $f_1$ and $f_2$ for DTMDP($g^*$) are given by,

$$P^1(f_1)=
\begin{matrix}
 \begin{pmatrix}
1 ~ & 0\\
0 ~ & 1
\end{pmatrix}
\end{matrix}, 
\hspace{.25cm}
 P^1(f_2)=  \begin{matrix}
 \begin{pmatrix}
0 ~ & 1\\
0 ~ & 1
       \end{pmatrix}
\end{matrix}.
$$
 We have
\begin{equation} \label{e1}
u_{0.6}^1(f_1)=[I-0.6 P(f_1)]^{-1}\bar{r}^1(f_1)=(6,4.5).
\end{equation}

\begin{equation} \label{e2}
u_{0.6}^1(f_2)=[I-0.6 P(f_2)]^{-1}\bar{r}^1(f_2)=(5.7,4.5).
\end{equation}
From \eqref{e1} and \eqref{e2} $f^1=(1,0)=f^*$ is the optimal policy of DTMDP($g^*$). Therefore, $f^*$ is the optimal 
policy of CTMDP($g^*$), i.e., $f^*$ is best response of $g^*$. Now, fix $f^*$, then player 2 faces CTMDP($f^*$).
To compute the optimal policy of CTMDP($f^*$) at $\alpha=\frac{2}{3}$, we solve the equivalent DTMDP($f^*$) defined by 
\eqref{pl2-eqdata} at $\beta=0.6$.
Let  $g_1=(1,0)$ and $g_2=(0,1)$ be two
stationary deterministic policies for player 2. Using the above data, the 
transition probability matrices induced by $g_1$ and $g_2$ for DTMDP($f^*$) are given by,

$$P^1(g_1)=
\begin{matrix}
 \begin{pmatrix}
1 ~ & 0\\
0 ~ & 1
\end{pmatrix}
\end{matrix}, 
\hspace{.25cm}
 P^1(g_2)=  \begin{matrix}
 \begin{pmatrix}
0 ~ & 1\\
0 ~ & 1
       \end{pmatrix}
\end{matrix}.
$$
We have
\begin{equation} \label{e3}
u_{0.6}^2(g_1)=[I-0.6 P(g_1)]^{-1}\bar{r}^2(g_1)=(6.6,6).
\end{equation}

\begin{equation} \label{e4}
u_{0.6}^2(g_2)=[I-0.6 P(g_2)]^{-1}\bar{r}^2(g_2)=(6.6,6).
\end{equation} 
From \eqref{e3} and \eqref{e4}, $g_1$ and $g_2$ both are the optimal policies of DTMDP($f^*$) at $\beta=0.6$. 
This implies $g^*=g_1$ is the best response of $f^*$ at $\alpha=\frac{2}{3}$. Hence $(f^*,g^*)$ is a Nash equilibrium  
at $\alpha=\frac{2}{3}$.
It is easy to check that all the conditions 
\textbf{N1},
\textbf{N2},
\textbf{N3} are satisfied at $(f^*,g^*)$. Hence, from Theorem \ref{Blackwell_disc_cont_thm2} $(f^*,g^*)$ is a BNE.
From \eqref{Blackwell-bound}, $\alpha_0=\frac{2}{3}$, i.e., $(f^*,g^*)$ is a Nash equilibrium for all $\alpha\in \left(0,\frac{2}{3}\right]$.

\section{Conclusions}\label{Conclusions}
We study BNE in both discrete and continuous time stochastic games.
We give counter examples to show that in general discrete as well as continuous time stochastic games
need not possess a stationary BNE. 
We show the existence of a BNE for SC-AR stochastic games. For general stochastic games 
we give two different sets of conditions that together are sufficient for a Nash equilibrium
to be a BNE. We give few examples which show that the Nash equilibria satisfying the proposed sufficient conditions
indeed exist.

\end{document}